\documentclass[aihp]{imsart}

\RequirePackage{amsthm,amsmath,amsfonts,amssymb}
\RequirePackage[numbers]{natbib}
\RequirePackage[colorlinks,citecolor=blue,urlcolor=blue]{hyperref}
\usepackage{csquotes}
\usepackage{dsfont}
\usepackage{graphicx}

\startlocaldefs
\theoremstyle{plain}
  \newtheorem{thm}{Theorem}[section]
  \newtheorem{lem}[thm]{Lemma}
  
  \newtheorem{cor}[thm]{Corollary} 
\theoremstyle{definition}

\theoremstyle{plain}

\def\dd{\mathrm{d}}

\def\CF{\mathcal{F}}
\def\CG{\mathcal{G}}

\def\CZ{\mathcal{Z}}
\def\1{\mathds{1}}

\def\BE{\mathbb{E}}

\def\BN{\mathbb{N}}
\def\BR{\mathbb{R}}
\def\BP{\mathbb{P}}

\def\Var{\text{Var}}
\def\Cov{\text{Cov}}

\newcommand{\indep}{\perp\!\!\!\!\!\!\perp} 

\endlocaldefs

\begin{document}

\begin{frontmatter}

\title{Local estimation of transition rates of jump processes through discretization}
\runtitle{Local estimation of transition rates of jump processes through discretization}

\begin{aug}
\author[A]{\inits{}\fnms{Martin}~\snm{Bladt}\ead[label=e1]{martinbladt@math.ku.dk}} \and
\author[A]{\inits{}\fnms{Rasmus Frigaard}~\snm{Lemvig}\ead[label=e2]{rfl@math.ku.dk}}
\address[A]{Department of Mathematical Sciences, University of Copenhagen, Copenhagen, 2100, DK\printead[presep={,\ }]{e1}, \printead[presep={,\ }]{e2}}

\end{aug}

\begin{abstract}
We investigate the Poisson regression method for Markov and semi-Markov jump processes from a nonparametric angle, allowing the lengths of the time and duration intervals in the partition to vary with the number of observations. Imposing no structural assumptions on the true intensities, we obtain asymptotic normality of the occurence/exposure rates under appropriate shrinking conditions on the partition lengths. We derive asymptotic normality results for both Markov and semi-Markov models using only classical central limit theorems and elementary results for counting processes. All results are illustrated on both simulated and real data. 
\end{abstract}




\end{frontmatter}
\tableofcontents
\newpage

\section{Introduction}

In this paper, we consider a novel approach to the problem of transition rate estimation for jump processes using Poisson regression. Poisson regression goes back to \cite{PrenticeGloeckler}, \cite{Holford} and \cite{LairdOlivier} (see also \cite{Aalenetal} for background) and is a simple method which is easy to communicate and fast to fit, making it a popular choice among practitioners. For actuaries, the Poisson regression method is especially significant since it allows for the efficient analysis of large-scale insurance portfolios while maintaining the transparency and interpretability required by regulators and stakeholders. Poisson regression arises from the assumption that transition intensities are piecewise constant, making the method inherently parametric, but still with a high degree of flexibility, striking a compromise between nonparametric methods and methods imposing strong structural assumptions on the data.

The classical approach for establishing consistency and asymptotic normality for Poisson regression relies on general results for maximum likelihood estimators for counting process models (see section 5.1 of \cite{Borgan} for the case of piecewise constant intensities and \cite{ABGK} for more background), and the assumptions behind these results are often unverifiable and highly technical in nature. We present a new approach for establishing consistency and asymptotic normality with very mild assumptions on the data-generating process and easy-to-implement requirements on the partition interval lengths. We provide results for Markov and semi-Markov jump processes, respectively. Markov multi-state models constitute the backbone of classical life insurance models and go back to \cite{HoemMarkov}, \cite{HoemCongress} and \cite{Amsler}, and Markov models remain important from both practical and theoretical perspectives, see \cite{Koller} for a modern treatment. A reason for the success of Markov models is their simplicity in regards to computation of transition probabilities and quantities of interest in life insurance such as cash flows and reserves (see Chapter V of \cite{SteffensenAsmussen} for background on Markov models in life insurance). Semi-Markov models are a generalisation of Markov models that allow the behaviour of the process to depend on the duration in the current state. This, in contrast to the classical Markov jump process, allows for effects that are specific to duration. Examples include disability insurance and loss of earning capacity, where the reactivation intensity can be significantly larger at the onset of disability. Applications of semi-Markov processes in life insurance go back to \cite{Janssen} and \cite{HoemSemiMarkov} and are investigated in depth by \cite{Helwich}. A modern treatment with extensions to policyholder behaviour is given in \cite{BuchardtMoller}. For theoretical background on semi-Markov processes, consult e.g. \cite{JanssenDominicis} and \cite{Nollau}.


In this paper, we present a new nonparametric approach to Poisson regression for Markov and semi-Markov models. Nonparametric approaches for transition rate estimation are plentiful, dating back to \cite{Nelson} and \cite{Aalen} and later developed by, among others, \cite{Beran}, \cite{Dabrowska} and \cite{BladtFurrer}. Methods based on kernel smoothing were introduced in \cite{RamlauHansen} and further investigated in \cite{NielsenLinton}, \cite{DabrowskaCox} and \cite{McKeagueUtikal}. Our contribution consists in a new analysis of the already established Poisson regression methods for Markov and semi-Markov models. Instead of letting the time (and for semi-Markov models, duration) intervals be fixed under the assumption that the true intensity is piecewise constant on this partition, we impose no structural assumptions on the true intensities and let the interval lengths shrink at a proper rate. Both the asymptotic normality results themselves as well as their proofs are elementary in nature, relying only on results from classical probability theory and elementary considerations from counting process theory. For a review of the tools we need from the latter, consult Appendix \ref{sec:background}.

The article is structured as follows. In Section \ref{sec:Markov}, we discuss the classical smooth Markov model, starting with model specification and recapping the Poisson regression method before moving on to the asymptotic normality results for dynamic bin widths. Section \ref{sec:semiMarkov} is built up the same way, but for the semi-Markov model. Section \ref{sec:simulation} contains simulation studies for both Markov and semi-Markov models that illustrate our theoretical results, while Section \ref{sec:Application} showcases the results on a real dataset. Section \ref{sec:Discussion} concludes the paper with a discussion and suggestions for future work. Throughout the article, proofs of lemmata and corollaries are deferred to Appendix \ref{sec:Proofs}, while proofs of the main results, namely Theorems \ref{thm:PoissonMarkov} and \ref{thm:PoissonSemiMarkov}, are included in the main text.

We fix a background filtered probability space $(\Omega, \CF, (\CF_t)_{t \geq 0}, \BP)$ satisfying the usual conditions. Throughout, we consider a jump process $Z = (Z_t)_{t \geq 0}$ on a finite state space $\mathcal{Z}$. Let $N$ denote the multivariate counting process with components $N_{jk} = (N_{jk}(t))_{t \geq 0}$ counting the number of jumps between states,
\begin{equation*}
    N_{jk}(t) = \#\{s \in (0, t] : Z_{s-} = j, Z_s = k\}
\end{equation*}
for $j, k \in \mathcal{Z}, j \neq k$. We assume
\begin{equation}\label{eq:finitefirstmoment}
    \BE[N_{jk}(t)] < \infty, \quad t \geq 0
\end{equation}
which also guarantees that $Z$ is non-explosive. We also impose the following assumption referred to as \emph{random right-censoring}, namely
\begin{equation}\label{eq:independentRightCensoring}
    Z \indep R.
\end{equation}
All samples throughout the article are assumed to be iid. All limits are understood to be for $n \to \infty$ unless otherwise stated.

\section{The Markov model}\label{sec:Markov}

\subsection{Model specification}

The Markov model can be specified using \emph{cumulative transition rates} $q_{jk}$ via the Doob--Meyer decomposition of $N_{jk}$ by
\begin{equation*}
    N_{jk}(t) = M_{jk}(t) + \int_{0}^t 1_{(Z_{s-} = j)}q_{jk}(\dd s),
\end{equation*}
where $M_{jk}$ is a mean-zero martingale with respect to $(\CF_t)_{t \geq 0}$, assumed to be (the completion of) the natural filtration generated by $Z$. In the context of Poisson regression, one assumes a \emph{smooth} Markov model where 
\begin{equation*}
    q_{jk}(t) = \int_0^t \mu_{jk}(s)\dd s
\end{equation*}
with the $\mu_{jk}$ measurable and sufficiently regular. We refer to the $\mu_{jk}$ as \emph{transition rates} or \emph{intensities} for the Markov process. The interpretation of the intensity $\mu_{jk}$ is as follows. In the limit $h \to 0$, it holds for $t \geq 0$ that
\begin{equation*}
    \BP(Z_{t + h} = k \mid Z_t = j) = \delta_{jk} + \mu_{jk}(t)h + o(h)
\end{equation*}
with $\delta_{jk} = 1_{(j = k)}$ the Kronecker delta, so that if $j \neq k$, we may interpret $\mu_{jk}(t)dt$ as the probability of jumping from state $j$ to $k$ in the infinitesimal time interval $[t, t + dt)$ given occupancy in $j$ at time $t$. In the presence of censoring, we let $(\CG_t)_{t \geq 0}$ denote the completion of the enlarged filtration
\begin{equation*}
    t \mapsto \CF_t \lor \sigma(C(s) : s \leq t)
\end{equation*}
where $C(s) = 1_{(s \leq R)}$ is the censoring process, see also chapter III.2.2 of \cite{ABGK}. Note that $C$ is predictable with respect to the $\CG$-filtration, and by the random right-censoring assumption (\ref{eq:independentRightCensoring}), it follows from Examples III.2.5 and III.2.8 in \cite{ABGK} that the compensators of $N_{jk}$ are unchanged when passing to the larger filtration $\CG$ and that the $\CG$-compensator of $t \mapsto N_{jk}(t \land R)$ is given by
\begin{equation*}
    t \mapsto \int_0^t 1_{(s \leq R)}1_{(Z_{s-} = j)}\mu_{jk}(s)\dd s.
\end{equation*}
This will play a role in some of the computations below. In the main result, the \emph{occupation probability} under censoring
\begin{equation*}
    p_j^\texttt{c}(t) := \BP(Z_t = j, t \leq R)
\end{equation*}
is an essential quantity. 

\subsection{Poisson regression for the Markov model}

We briefly recall the background for the Poisson regression method for a smooth Markov model. To start with, consider the general setup of parameter-dependent intensities $\theta \mapsto \mu_{jk}(\cdot \mid \theta)$. Letting $\pi$ denote the vector of initial probabilities
\begin{equation*}
    \pi = (\pi_j)_{j \in \mathcal{Z}}, \quad \pi_j = \BP(Z_0 = j),
\end{equation*}
the right-censored log-likelihood becomes (under the random right-censoring assumption)
\begin{equation*}
    \ell_n(\pi, \theta) = \sum_{j \in \mathcal{Z}} B_j \log \pi_j + \sum_{i = 1}^n \sum_{j \neq k} \left(\int_0^{R^i} \log \mu_{jk}(s \mid \theta) \dd N_{jk}^i(s) - \int_0^{R^i} I_j^i(s)\mu_{jk}(s \mid \theta) \dd s\right),
\end{equation*}
where $B_j = \sum_{i = 1}^n I_j^i(0)$ and $I_j^i(s) = 1_{(Z_s^i = j)}$. Assuming that the $\mu_{jk}$ are piecewise constant on the grid $0 = t_0 < t_1 < \cdots < t_M = \infty$ with
\begin{equation*}
    \mu_{jk}(t \mid \theta) = \mu_{jk}^{(m)}, \quad t \in [t_{m - 1}, t_m),
\end{equation*}
the log-likelihood $\ell_n$ above simplifies to
\begin{equation*}
    \ell_n(\pi, \theta) =\sum_{j \in \mathcal{Z}}B_j \log \pi_j +  \sum_{m = 1}^M \sum_{j \neq k} \Big(O_{jk}(m) \log \mu_{jk}^{(m)} - E_j(m) \mu_{jk}^{(m)}\Big),
\end{equation*}
where $O_{jk}(m)$ denotes the number of jumps from state $j$ to state $k$ in the interval $[t_{m - 1}, t_m)$ (the \emph{occurence}), and $E_j(m)$ is the total time spent in $[t_{m - 1}, t_m)$ (the \emph{exposure}):
\begin{equation*}
    O_{jk}(m) = \sum_{i = 1}^n \int_{t_{m - 1}}^{t_m} \dd N_{jk}^i(s \land R^i), \quad E_j(m) = \sum_{i = 1}^n \int_{t_{m - 1}}^{t_m} 1_{(s < R^i)} I_j^i(s)\dd s.
\end{equation*}
The parameter $\theta$ now consists of the $M$ values $\mu_{jk}^{(m)}$, and we want to solve for the MLE of $\theta$ and $\pi$. Recognising that the log-likelihood $\ell_n(\pi, \theta)$ is (up to addition with a constant) the same as one would obtain from independent observations
\begin{equation*}
    (B_j)_{j \in \mathcal{Z}}, \quad (O_{jk}(m))_{j, k \in \mathcal{Z}, j \neq k, m = 1, \dots, M}
\end{equation*}
with $(B_j)_{j \in \mathcal{Z}}$ multinomial with parameters $n$ and $\pi$ and $O_{jk}(m)$ Poisson with parameter $E_j(m) \mu_{jk}^{(m)}$ (considering $n$ and $E_j(m)$ fixed), the MLE's are readily seen to be
\begin{equation*}
    \hat{\pi}_j = \frac{B_j}{n} \quad \text{and} \quad \hat{\mu}_{jk}^{(m)}  = \frac{O_{jk}(m)}{E_j(m)},
\end{equation*}
and we refer to the $\hat{\mu}_{jk}^{(m)}$ as \emph{occurence/exposure rates} or simply \emph{OE rates}. In the rest of the present work, we focus solely on the OE rates. Computation of the $\hat{\mu}_{jk}^{(m)}$ are typically carried out by fitting a Poisson GLM with log-link function and log-exposure as offsets.

\subsection{Main results}\label{sec:main}

When fitting a Poisson regression model in practice, one chooses a partition for the hazard function upfront, and one assumes that the true hazard function is indeed constant on each of the subintervals. From a theoretical point of view, this allows asymptotic properties to be established via general results for maximum likelihood estimators. We refer to Chapter VI.1.2 of \cite{ABGK} and their Example VI.1.1 for background. Our approach is different. Instead of imposing assumptions on the form of the true hazard function $\mu_{jk}$, we assume that the intervals $I_n = [t_{m - 1}, t_m)$ in the partition are dynamic with lengths shrinking to zero at a proper rate. The following theorem provides an asymptotic normality result for the estimator $\hat{\mu}^{(m)}_{jk}$ of $\mu_{jk}(t)$ when $t \in [t_{m - 1}, t_m)$. 

\begin{thm}[Asymptotic normality (Markov)]\label{thm:PoissonMarkov}
Let $Z$ be a Markov process with intensities $\mu_{jk}$. Fix $j, k \in \mathcal{Z}$, $j \neq k$, and $t$. Assume $\mu_{jk}(t) > 0$, $\BE[N_{jk}(t)^4] < \infty$ for $t \geq 0$ and let $m = m_n$ be such that $t \in [t_{m - 1}, t_m)$. Assume $p_j^\texttt{c}$ and $\mu_{jk}$ are $C^1$ and $C^2$, respectively, in a neighbourhood of $t$ and $p_j^\texttt{c}(t) > 0$. Let $\Delta_n := t_m - t_{m - 1} \downarrow 0$, $n\Delta_n \to \infty$ and $n\Delta_n^3 \to 0$. Then
\begin{equation*}
	\sqrt{n \Delta_n}(\hat{\mu}^{(m)}_{jk} - \mu_{jk}(t)) \overset{d}{\longrightarrow} N\Big(0, \frac{\mu_{jk}(t)}{p_j^\texttt{c}(t)} \Big).
\end{equation*}
\end{thm}

Before embarking on the proof, we introduce the notation
\begin{equation*}
    X_{i, n} := \int_{I_n} N_{jk}^i(\dd s \land R^i), \quad Y_{i, n} := \int_{I_n} 1_{(s < R^i)}1_{(Z_{s-}^i = j)}\dd s
\end{equation*}
so that
\begin{equation*}
    O_{jk}(m) = \sum_{i = 1}^n X_{i, n} \quad \text{and} \quad E_j(m) = \sum_{i = 1}^n Y_{i, n}. 
\end{equation*}
We also let $\mu_n^{jk} = \BE[X_{1, n}]$ and $\eta_n^j = \BE[Y_{1, n}]$. The following lemma collects some elementary results for $X_{i, n}$ and $Y_{i, n}$.

\begin{lem}\label{lem:meanvarianceMarkov}
Under the assumptions of Theorem \ref{thm:PoissonMarkov}, the following holds. 
\begin{enumerate}
    \item $\mu_n^{jk} = \int_{I_n} p_j^\texttt{c}(s)\mu_{jk}(s)\dd s = \Delta_n p_j^\texttt{c}(t) \mu_{jk}(t) + O(\Delta_n^2)$.
    \item $\eta_n^j = \int_{I_n} p_j^\texttt{c}(s)\dd s = \Delta_n p_j^\texttt{c}(t) + O(\Delta_n^2)$.
    \item If $\BE[N_{jk}(t)^2] < \infty$ for all $t \geq 0$, we have $\operatorname{Var}(X_{1, n}) = \Delta_n p_j^\texttt{c}(t) \mu_{jk}(t) + O(\Delta_n^2)$.
    \item $\operatorname{Var}(Y_{1, n}) = \Delta_n^2p_j^\texttt{c}(t) + O(\Delta_n^2)$.
\end{enumerate}
\end{lem}

The following technical lemma is a crude but elementary way to establish a Lindeberg condition for the occurences. The proof of this result and the previous lemma can be found in the appendix.

\begin{lem}[Lyapunov condition (Markov)]\label{lem:Lyapunov}
If $\BE[N_{jk}(t)^4] < \infty$, $\Delta_n \downarrow 0$ and $n\Delta_n \to \infty$, the Lyapunov condition of order 4 holds for the triangular array $W_{i,n}$ given by
\begin{equation*}
    W_{i,n} = \frac{X_{i, n} - \mu_n^{jk}}{\sqrt{\operatorname{Var}(\sum_{i = 1}^n X_{i, n})}}.
\end{equation*}
\end{lem}


\begin{proof}[Proof of Theorem~\ref{thm:PoissonMarkov}]
We have
\begin{equation*}
	\frac{O_{jk}(m) - n\mu_n^{jk}}{\sqrt{\operatorname{Var}(\sum_{i = 1}^n X_{i, n})}} = \sum_{i = 1}^n \frac{X_{i, n} - \mu_n^{jk}}{\sqrt{\operatorname{Var}(\sum_{i = 1}^n X_{i, n})}},
\end{equation*}
and the terms in this sum are independent with mean-zero. By construction, the sum of the variances equals 1. By Lemma~\ref{lem:Lyapunov}, the Lindeberg condition holds and we can thus conclude by Lindeberg's CLT, Slutsky's lemma and Lemma~\ref{lem:meanvarianceMarkov} that
\begin{equation*}
	\frac{O_{jk}(m) - n\mu_n^{jk}}{\sqrt{n\Delta_n}} \overset{d}{\to} N(0, \mu_{jk}(t)p_j^\texttt{c}(t)).
\end{equation*}
We now claim that
\begin{equation*}
	\frac{E_j(m) - n\eta_n^j}{\sqrt{\operatorname{Var}(\sum_{i = 1}^n Y_{i, n})}} = O_\BP(1)
\end{equation*}
by showing convergence in distribution to $N(0, 1)$. We have
    \begin{equation*}
        \frac{E_j(m) - n\eta_n^j}{\sqrt{\operatorname{Var}(\sum_{i = 1}^n Y_{i, n})}} = \sum_{i = 1}^n \frac{Y_{i, n} - \eta_n^j}{\sqrt{\operatorname{Var}(\sum_{i = 1}^n Y_{i, n})}} =: \sum_{i = 1}^n A_{i, n}
    \end{equation*}
    and the $A_{i, n}$ are independent with mean-zero. We can verify Lindeberg's condition directly using Lemma \ref{lem:meanvarianceMarkov} as follows. Let $c > 0$ be arbitrary. Then if we define $s_n = \operatorname{Var}(\sum_{i = 1}^n Y_{i, n})$, we get from Chebyshev's inequality that
    \begin{align*}
        \BE[A_{i, n}^2 1_{(|A_{i, n}| > c)}] &= \frac{\BE\Big[(Y_{i, n} - \eta_n^j)^2 1_{(|Y_{i, n} - \eta_n^j| > c\sqrt{s_n})}\Big]}{s_n} \\
        &\leq \frac{4\Delta_n^2 \BP(|Y_{i, n} - \eta_n^j| > c\sqrt{s_n})}{s_n} \\
        &\leq \frac{4}{c^2} \cdot \frac{\Delta_n^2 \Var(Y_{1, n})}{s_n^2} = O\left(\frac{\Delta_n^4}{n^2 \Delta_n^4} \right) \\
        &= O\left(\frac{1}{n^2}\right),
    \end{align*}
    so that
    \begin{equation*}
        \sum_{i = 1}^n \BE[A_{i, n}^2 1_{(|A_{i, n}| > c)}] = O\left(\frac{1}{n}\right) \to 0.
    \end{equation*}
    From Lindeberg's CLT, we have
    \begin{equation*}
        \frac{E_j(m) - n\eta_n^j}{\sqrt{\operatorname{Var}(\sum_{i = 1}^n Y_{i, n})}} \overset{d}{\to} N(0, 1)
    \end{equation*}
    and so the left hand side is bounded in probability as claimed. Thus,
\begin{align*}
    E_j(m) &= \sum_{i = 1}^n (Y_{i, n} - \eta_n^j) + n\eta_n^j = \sqrt{\operatorname{Var}\Big(\sum_{i = 1}^n Y_{i, n}\Big)} \sum_{i = 1}^n \frac{Y_{i, n} - \eta_n^j}{\sqrt{\operatorname{Var}(\sum_{i = 1}^n Y_{i, n})}} + n \eta_n^j \\
    &= \sqrt{n\Delta_n^2p_j^\texttt{c}(t)(1 + O(1))} O_\BP(1) + n\Delta_n (p_j^\texttt{c}(t) + O(\Delta_n)) \\
    &= n\Delta_n (p_j^\texttt{c}(t) + o_\BP(1)).
\end{align*}
We have
\begin{equation*}
    \hat{\mu}_{jk}^{(m)} - \mu_{jk}(t) = \frac{O_{jk}(m) - n\mu_n^{jk}}{E_j(m)} + \frac{n\mu_n^{jk} - \mu_{jk}(t) E_j(m)}{E_j(m)} =: (1) + (2).
\end{equation*}
Term (1) multiplied with $\sqrt{n\Delta_n}$ converges to the desired asymptotic distribution by the asymptotic expression for $E_j(m)$ and a Slutsky argument. It remains to show that (2) goes to zero in probability. We can further decompose
\begin{equation*}
    (2) = \frac{n(\mu_n^{jk} - \mu_{jk}(t)\eta_n^j)}{E_j(m)} - \frac{\mu_{jk}(t)(E_j(m) - n\eta_n^j)}{E_j(m)} =: (i) - (ii).
\end{equation*}
Since $\operatorname{Var}(E_j(m) - n \eta_n^j) = O(n \Delta_n^2)$, an application of Chebyshev's inequality yields $E_j(m) - n\eta_n^j = O_\BP(\sqrt{n} \Delta_n)$ and so
\begin{equation*}
    \sqrt{n \Delta_n} (ii) = O_\BP\Big(\sqrt{n \Delta_n} \frac{\sqrt{n} \Delta_n}{n\Delta_n} \Big) = O_\BP(\Delta_n)
\end{equation*}
which goes to zero. As for (i),
\begin{equation*}
    \sqrt{n \Delta_n}(i) = O_\BP\Big(\sqrt{n \Delta_n} \frac{n \Delta_n^2}{n \Delta_n} \Big) = O_\BP(\sqrt{n \Delta_n^3})
\end{equation*}
which also goes to zero by assumption. All in all, $\sqrt{n \Delta_n} (2)$ goes to zero in probability, and the proof is complete. 
\end{proof}

If we choose distinct timepoints $s \neq t$, the time intervals eventually do not overlap due to the assumption $\Delta_n \to 0$. Hence we expect the occurence/exposure rates for the intervals containing $s$ and $t$ to become asymptotically independent. This is indeed the case as the following corollary shows. Since $t$ has so far been considered fixed, we have suppressed the dependence on $t$ in the notation for the interval of $t$. To emphasise this dependence, we write $m(t), X_{i, n}(t), Y_{i, n}(t), \mu_n^{jk}(t), \eta_n^j(t)$ and $I_n(t)$ in the following. As the partition is assumed equidistant, $\Delta_n$ is independent of $t$.

\begin{cor}[Asymptotic independence for distinct timepoints (Markov)]\label{cor:asymptoticIndependenceMarkov}
Under the assumptions in Theorem \ref{thm:PoissonMarkov}, it holds for $t, s \geq 0$ with $t \neq s$ and $p_j^\texttt{c}(t), p_j^\texttt{c}(s) > 0$ that
\begin{equation*}
    \sqrt{n\Delta_n}(\hat{\mu}^{m(t)} - \mu_{jk}(t), \hat{\mu}^{m(s)} - \mu_{jk}(s)) \overset{d}{\longrightarrow} (G(t), G(s))
\end{equation*}
with $G(t)$ and $G(s)$ independent and
\begin{equation*}
    G(t') \sim N\Big(0, \frac{\mu_{jk}(t')}{p_j^\texttt{c}(t')}\Big).
\end{equation*}
\end{cor}

\section{The semi-Markov model}\label{sec:semiMarkov}

\subsection{Model specification}

We say that $Z$ is semi-Markov if the joint process $(Z, U)$ is Markov, where $U$ denotes the \emph{duration process} given by
\begin{equation*}
    U_t = \inf\{r > 0 : Z_{t - r} \neq Z_t\},
\end{equation*}
that is, the time since the previous jump. To specify $Z$, we again use the compensators of the counting processes $N_{jk}$. We assume that these are given by
\begin{equation*}
    t \mapsto \int_0^t 1_{(Z_{s-} = j)}\mu_{jk}(s, U_{s-})\dd s
\end{equation*}
where the $\mu_{jk}$ are measurable and sufficiently regular (duration-dependent) transition intensities. The exact requirements on the $\mu_{jk}$ will be specified along the way. As in the Markov case, it holds for all $t, u \geq 0$ and $j, k \in \mathcal{Z}$ that
\begin{equation*}
    \BP(Z_{t + h} = k \mid Z_t = j, U_t = u) = \delta_{jk} + \mu_{jk}(t, u)h + o(h).
\end{equation*}
In the following, we consider a censored occupation probability which also takes duration into account, namely
\begin{equation*}
    p_j^\texttt{c}(t, u) := \BP(Z_t = j, U_t \leq u, t \leq R).
\end{equation*}
In the following, the partial derivative of $p_j^\texttt{c}$ with respect to duration will play a prominent role. We denote this derivative in the point $(t, u)$ (when it exists) by $\partial_2 p_j^\texttt{c}(t, u)$. All the filtration considerations are the same as in the Markov case.

\subsection{Poisson regression for the semi-Markov model}

As in the Markov case, we briefly cover the background for the Poisson regression method for semi-Markov processes. Disregarding the initial probabilities, the right-censored log-likelihood $\ell_n(\theta)$ becomes
\begin{equation*}
    \ell_n(\theta) = \sum_{i = 1}^n \sum_{j \neq k} \left(\int_0^{R^i} \log \mu_{jk}(s, U^i_{s-} \mid \theta) \dd N_{jk}^i(s) - \int_0^{R^i} I_j^i(s) \mu_{jk}(s, U^i_{s-} \mid \theta)\dd s \right).
\end{equation*}
Assuming that $\mu_{jk}(s, v \mid \theta)$ is piecewise constant on a two-dimensional grid $[t_{m_1 - 1}, t_{m_1}) \times [u_{m_2 - 1}, u_{m_2})$,
\begin{equation*}
    \mu_{jk}(s, v \mid \theta) = \mu_{jk}^{(m_1, m_2)}, \quad (s, v) \in [t_{m_1 - 1}, t_{m_1}) \times [u_{m_2 - 1}, u_{m_2}),
\end{equation*}
the log-likelihood simplifies to
\begin{equation*}
    \ell_n(\theta) = \sum_{m_1, m_2} \sum_{i = 1}^n \Big(O_{jk}(m_1, m_2) \log \mu_{jk}^{(m_1, m_2)} - E_j(m_1, m_2) \mu_{jk}^{(m_1, m_2)} \Big)
\end{equation*}
with $O_{jk}(m_1, m_2)$ and $E_j(m_1, m_2)$ again denoting occurence and exposure in the box $[t_{m_1 - 1}, t_{m_1}) \times [u_{m_2 - 1}, u_{m_2})$,
\begin{align*}
    O_{jk}(m_1, m_2) &= \sum_{i = 1}^n \int_0^{R^i} 1_{(t_{m_1 - 1} \leq s < t_{m_1})} 1_{(u_{m_2 - 1} \leq U^i_{s-} < u_{m_2})} \dd N_{jk}^i(s), \\
    E_j(m_1, m_2) &= \sum_{i = 1}^n \int_0^{R^i} 1_{(t_{m_1 - 1} \leq s < t_{m_1})} 1_{(u_{m_2 - 1} \leq U^i_{s-} < u_{m_2})} 1_{(Z_{s-}^i = j)}\dd s.
\end{align*}
By the same arguments as in the Markov case, the MLE of $\mu_{jk}^{(m_1, m_2)}$ is given by the OE rate
\begin{equation*}
    \hat{\mu}_{jk}^{(m_1, m_2)} = \frac{O_{jk}(m_1, m_2)}{E_j(m_1, m_2)}.
\end{equation*}

\subsection{Main results}

In the Poisson regression for a semi-Markov process, we make a partition into boxes $I_n^{(1)} \times I_n^{(2)}$ with $I_n^{(1)} := [t_{m_1 - 1}, t_{m_1})$ the time interval and $I_n^{(2)} := [u_{m_2 - 1}, u_{m_2})$ the duration interval. We let $\Delta_n^{(1)} := t_{m_1} - t_{m_1 - 1}$ and $\Delta_n^{(2)} := u_{m_2} - u_{m_2 - 1}$. 

\begin{thm}[Asymptotic normality (semi-Markov)]\label{thm:PoissonSemiMarkov}
Let $Z$ be a semi-Markov process with intensities $\mu_{jk}$. Fix $j, k \in \mathcal{Z}$, $j \neq k$, a time $t$ and a duration $u$ with $t > u$. Assume $\mu_{jk}(t, u) > 0$, $\BE[N_{jk}(t)^4] < \infty$ for all $t \geq 0$ and let $(m_1, m_2) = (m_{1, n}, m_{2, n})$ be such that $(t, u) \in I_n^{(1)} \times I_n^{(2)}$. Assume $p_j^c$ and $\mu_{jk}$ are $C^2$ in a neighbourhood of $(t, u)$ with $\partial _2p_j^\texttt{c}(t, u) > 0$. Under the restrictions
\begin{enumerate}
    \item[(i)] $\Delta_n^{(1)} \to 0$ and $\Delta_n^{(2)} \to 0$,
    \item[(ii)] $n\Delta_n^{(1)}\Delta_n^{(2)} \to \infty$,
    \item[(iii)] $n\Delta_n^{(1)}\big(\Delta_n^{(2)} \big)^3 \to 0$ and $n\big(\Delta_n^{(1)}\big)^3 \Delta_n^{(2)} \to 0$,
\end{enumerate}
it holds that
\begin{equation*}
    \sqrt{n\Delta_n^{(1)}\Delta_n^{(2)}}\Big(\hat{\mu}_{jk}^{(m_1, m_2)} - \mu_{jk}(t, u)\Big) \overset{d}{\longrightarrow} N\left(0, \frac{\mu_{jk}(t, u)}{\partial_2 p_j^c(t, u)}\right).
\end{equation*}
\end{thm}

We again introduce the notation
\begin{align*}
    X_{i, n} &:= \int_0^{R^i} 1_{(t_{m_1 - 1} \leq s < t_{m_1})}1_{(u_{m_2 - 1} \leq U_{s-}^i < u_{m_2})}\dd N_{jk}^i(s), \\
    Y_{i, n} &:= \int_0^{R^i}1_{(t_{m_1 - 1} \leq s < t_{m_1})}1_{(u_{m_2 - 1} \leq U_{s-}^i < u_{m_2})} 1_{(Z_{s-}^i = j)}\dd s,
\end{align*}
so that
\begin{equation*}
    O_{jk}(m_1, m_2) = \sum_{i = 1}^n X_{i, n}, \quad E_j(m_1, m_2) = \sum_{i = 1}^n Y_{i, n}.
\end{equation*}
Define also $\mu_n^{jk} := \BE[X_{1, n}]$ and $\eta_n^j = \BE[Y_{1, n}]$. The following lemma yields results akin to those in Lemma \ref{lem:meanvarianceMarkov} for the Markov setup.

\begin{lem}\label{lem:meanvariancesemiMarkov}
Under the assumptions in Theorem \ref{thm:PoissonSemiMarkov}, the following holds for the semi-Markov setup.
\begin{enumerate}
    \item $\mu_n^{jk} = \Delta_n^{(1)} \Delta_n^{(2)} \mu_{jk}(t, u) \partial_2p_j^\texttt{c}(t, u) + O\Big(\Big(\Delta_n^{(1)}\big)^2 \Delta_n^{(2)}\Big) + O\Big(\Delta_n^{(1)}\big(\Delta_n^{(2)}\big)^2\Big)$.
    \item $\eta_n^j = \Delta_n^{(1)}\Delta_n^{(2)}\partial_2p_j^\texttt{c}(t, u) + O\Big(\big(\Delta_n^{(1)}\big)^2 \Delta_n^{(2)}\Big) + O\Big(\Delta_n^{(1)}\big(\Delta_n^{(2)}\big)^2\Big)$.
    \item $\operatorname{Var}(X_{1, n}) = \Delta_n^{(1)} \Delta_n^{(2)} \mu_{jk}(t, u) \partial_2p_j^\texttt{c}(t, u) + O\Big(\big(\Delta_n^{(1)}\big)^2 \Delta_n^{(2)}\Big) + O\Big(\Delta_n^{(1)}\big(\Delta_n^{(2)}\big)^2\Big)$.
    \item $\operatorname{Var}(Y_{1, n}) = \big(\Delta_n^{(1)}\big)^2 \Delta_n^{(2)} \partial_2 p_j^\texttt{c}(t, u) + O\Big(\big(\Delta_n^{(1)}\big)^2 \big(\Delta_n^{(2)}\big)^2\Big) + O\Big(\big(\Delta_n^{(1)}\big)^3 \Delta_n^{(2)}\Big)$.
\end{enumerate}
\end{lem}

\begin{lem}[Lyapunov condition (semi-Markov)]\label{lem:LyapunovSemiMarkov}
In the context of the semi-Markov model, if $\BE[N_{jk}(t)^4] < \infty$, $\Delta_n^{(1)}, \Delta_n^{(2)} \to 0$ and $n\Delta_n^{(1)}\Delta_n^{(2)} \to \infty$, the Lyapunov condition of order 4 holds for the triangular array $W_{i, n}$ given by
\begin{equation*}
    W_{i, n} = \frac{X_{i, n} - \mu_n^{jk}}{\sqrt{\operatorname{Var}(\sum_{i = 1}^n X_{i, n})}}.
\end{equation*}
\end{lem}

\begin{proof}[Proof of Theorem \ref{thm:PoissonSemiMarkov}]
From Lemma~\ref{lem:meanvariancesemiMarkov},
\begin{equation*}
    \operatorname{Var}\Big(\sum_{i = 1}^n X_{i, n} \Big) = n\Delta_n^{(1)}\Delta_n^{(2)}\Big(\mu_{jk}(t, u) \partial_2p_j^\texttt{c}(t, u) + O(\Delta_n^{(1)} + \Delta_n^{(2)}) \Big)
\end{equation*}
and from the Lindeberg condition established in Lemma~\ref{lem:LyapunovSemiMarkov}, we conclude that
\begin{equation*}
    \frac{O_{jk}(m_1, m_2) - n\mu_n^{jk}}{\sqrt{n \Delta_n^{(1)}\Delta_n^{(2)}}} \overset{d}{\longrightarrow} N\Big(0, \mu_{jk}(t, u) \partial_2p_j^\texttt{c}(t, u) \Big).
\end{equation*}
We again claim that
\begin{equation*}
    \sum_{i = 1}^n A_{i, n} := \sum_{i = 1}^n \frac{Y_{i, n} - \eta_n^j}{\sqrt{\operatorname{Var}(\sum_{i = 1}^n Y_{i, n})}} = \frac{E_j(m_1, m_2) - n\eta_n^{j}}{\sqrt{\operatorname{Var}(\sum_{i = 1}^n Y_{i, n} )}} \overset{d}{\longrightarrow} N(0, 1).
\end{equation*}
The $A_{i, n}$ are clearly independent with mean zero, and the variance of the sum is 1. We now verify the Lindeberg condition directly. Let $s_n = \operatorname{Var}\Big(\sum_{i = 1}^n Y_{i, n} \Big)$ and let $c > 0$ be fixed. Then
\begin{align*}
    \BE[A_{i, n}^2 1_{(|A_{i, n}| > c)}] &= \frac{\BE\Big[(Y_{i, n} - \eta_n^j)^2 1_{(|Y_{i, n} - \eta_n^j| > c \sqrt{s_n})}\Big]}{s_n} \\
    &\leq \frac{4\big(\Delta_n^{(1)}\big)^2 \BP(|Y_{i, n} - \eta_n^j| > c \sqrt{s_n})}{s_n} \\
    &\leq \frac{4}{c^2} \cdot \frac{\big(\Delta_n^{(1)}\big)^2 \Var(Y_{1, n})}{s_n^2} = O\left(\frac{\big(\Delta_n^{(1)}\big)^4 \Delta_n^{(2)}}{n^2 \big(\Delta_n^{(1)}\big)^4 \big(\Delta_n^{(2)}\big)^2} \right) \\
    &= O\left(\frac{1}{n} \cdot \frac{1}{n\Delta_n^{(2)}}\right),
\end{align*}
whence
\begin{equation*}
    \sum_{i = 1}^n \BE[A_{i, n}^2 1_{(|A_{i, n}| > c)}] = O\left(\frac{1}{n\Delta_n^{(2)}}\right) \to 0
\end{equation*}
as desired. Thus,
\begin{align*}
    E_j(m_1, m_2) &= \sqrt{\operatorname{Var}\Big(\sum_{i = 1}^n Y_{i, n} \Big)} \sum_{i = 1}^n A_{i, n} + n\eta_n^j \\
    &= \sqrt{n\big(\Delta_n^{(1)}\big)^2 \Delta_n^{(2)}\Big(\partial_2 p_j^\texttt{c}(t, u) + O(\Delta_n^{(1)} + \Delta_n^{(2)})\Big)} O_\BP(1) \\
    &+ n\Delta_n^{(1)}\Delta_n^{(2)} \partial_2p_j^\texttt{c}(t, u) + O\Big(n \big(\Delta_n^{(1)}\big)^2 \Delta_n^{(2)}\Big) + O\Big(n \Delta_n^{(1)} \big(\Delta_n^{(2)}\big)^2\Big).
\end{align*}
The first term is of order $\sqrt{n} \Delta_n^{(1)} \sqrt{\Delta_n^{(2)}}$ which is asymptotically smaller than $n\Delta_n^{(1)} \Delta_n^{(2)}$, and so
\begin{equation*}
    E_j(m_1, m_2) = n \Delta_n^{(1)}\Delta_n^{(2)}\Big(\partial_2p_j^\texttt{c}(t, u) + o_\BP(1)\Big). 
\end{equation*}
We make the same decomposition as for the Markov case and obtain
\begin{equation*}
    \hat{\mu}_{jk}^{(m_1, m_2)} -  \mu_{jk}(t, u) = \frac{O_{jk}(m_1, m_2) - n\mu_n^{jk}}{E_j(m_1, m_2)} + \frac{n\mu_n^{jk} - \mu_{jk}(t, u)E_j(m_1, m_2)}{E_j(m_1, m_2)} =: (1) + (2)
\end{equation*}
Term (1) multiplied with $n\Delta_n^{(1)}\Delta_n^{(2)}$ converges to the desired asymptotic distribution by the asymptotic expression for $E_j(m_1, m_2)$ and a Slutsky argument. We again decompose the second term into
\begin{equation*}
    (2) = \frac{n(\mu_n^{jk} - \mu_{jk}(t, u)\eta_n^j)}{E_j(m_1, m_2)} - \frac{\mu_{jk}(t, u)(E_j(m_1, m_2) - n\eta_n^j)}{E_j(m_1, m_2)} =: (i) - (ii).
\end{equation*}
Since $\operatorname{Var}(E_j(m_1, m_2) - n\eta_n^j) = O\Big(n \big(\Delta_n^{(1)}\big)^2 \Delta_n^{(2)}\Big)$, we have $E_j(m_1, m_2) - n\eta_n^j = O_\BP\Big(\sqrt{n} \Delta_n^{(1)} \sqrt{\Delta_n^{(2)}}\Big)$ and so
\begin{equation*}
    \sqrt{n \Delta_n^{(1)}\Delta_n^{(2)}}(ii) = O_\BP\left(\frac{\sqrt{n} \Delta_n^{(1)}\sqrt{\Delta_n^{(2)}}}{\sqrt{n \Delta_n^{(1)}\Delta_n^{(2)}}} \right) = O_\BP\Big(\sqrt{\Delta_n^{(1)}}\Big). 
\end{equation*}
Also,
\begin{align*}
    \sqrt{n\Delta_n^{(1)}\Delta_n^{(2)}}(i) &= O_\BP\left(\frac{n \big(\Delta_n^{(1)}\big)^2 \Delta_n^{(2)} + n \Delta_n^{(1)}\big(\Delta_n^{(2)}\big)^2}{\sqrt{n \Delta_n^{(1)}\Delta_n^{(2)}}} \right) \\
    &= O_\BP\Big(\sqrt{n \big(\Delta_n^{(1)}\big)^3 \Delta_n^{(2)}}\Big) + O_\BP\Big(\sqrt{n \Delta_n^{(1)} \big(\Delta_n^{(2)}\big)^3}\Big).
\end{align*}
By assumption, these two terms both converge to zero in probability, and the proof is complete.
\end{proof}

Just like for the Markov case, we can conclude asymptotic independence for distinct time points. Note that it is not a requirement that the duration intervals are distinct since the boxes will still be disjoint for different time intervals. We again make the dependence on time and duration explicit in our notation. 

\begin{cor}[Asymptotic independence for distinct timepoints (semi-Markov)]\label{cor:asymptoticIndependenceSemiMarkov}
Under the assumptions in Theorem \ref{thm:PoissonSemiMarkov}, it holds for $s, t \geq 0$ with $s \neq t$ and $u < t, v < s$ with $\partial_2p_j^\texttt{c}(t, u) > 0$ and $\partial_2 p_j^\texttt{c}(s, v) > 0$ that
\begin{equation*}
    \sqrt{n\Delta_n^{(1)}\Delta_n^{(2)}}(\hat{\mu}_{jk}^{(m_1(t), m_2(u))} - \mu_{jk}(t, u), \hat{\mu}_{jk}^{(m_1(s), m_2(v))} - \mu_{jk}(s, v)) \overset{d}{\longrightarrow} (G(t, u), G(s, v))
\end{equation*}
with $G(t, u)$ and $G(s, v)$ independent and
\begin{equation*}
    G(t', u') \sim N\Big(0, \frac{\mu_{jk}(t', u')}{\partial_2 p_j^\texttt{c}(t', u')} \Big).
\end{equation*}
\end{cor}

\section{Simulation study}
\label{sec:simulation}

The simulation study is included to illustrate the local character of the asymptotic results. We do not regard it as a justification of occurrence-exposure Poisson regression itself, which is classical. The focus is instead on pointwise estimation at a fixed interior time point and on the way the mesh width interacts with the sample size.

\subsection{Single-sample illustration}

We begin with a single simulated sample from a time-inhomogeneous Markov jump process on $\CZ = \{1,2,3\}$, where state~$3$ is absorbing and $Z_0 = 1$. The total jump intensity out of state~$i$ at time~$t$ is
\begin{align*}
\lambda_i(t) &=
\begin{cases}
0.1 + 0.002\, t + 0.05 \sin(t/2), & i = 1, \\
0.06 + 0.002\, t + 0.05 \sin(t/2), & i = 2, \\
0, & i = 3,
\end{cases}
\end{align*}
and, conditional on a jump from state~$i$, the destination state is selected according to fixed transition probabilities. From state~$1$, the process moves to state~$2$ with probability~$0.9$ and to state~$3$ with probability~$0.1$, whereas from state~$2$ all jumps go to state~$3$. Hence
\begin{align*}
\mu_{12}(t) &= 0.09 + 0.0018\, t + 0.045 \sin(t/2),
\end{align*}
which will also be the quantity considered in the later experiments. Independent right-censoring is imposed with $R \sim \mathrm{Unif}(10,40)$, and the transition intensities are estimated by occurrence-exposure Poisson regression on an equidistant grid on $[0,40]$.

Figure~\ref{fig:single_sample} is based on a single sample of size $n = 100{,}000$. It is included only as a reference plot. We observe that the fitted step functions track the true intensities reasonably well, but the figure is not intended to say anything by itself about the asymptotic regime.

\begin{figure}[t]
\centering
\includegraphics[width=\textwidth]{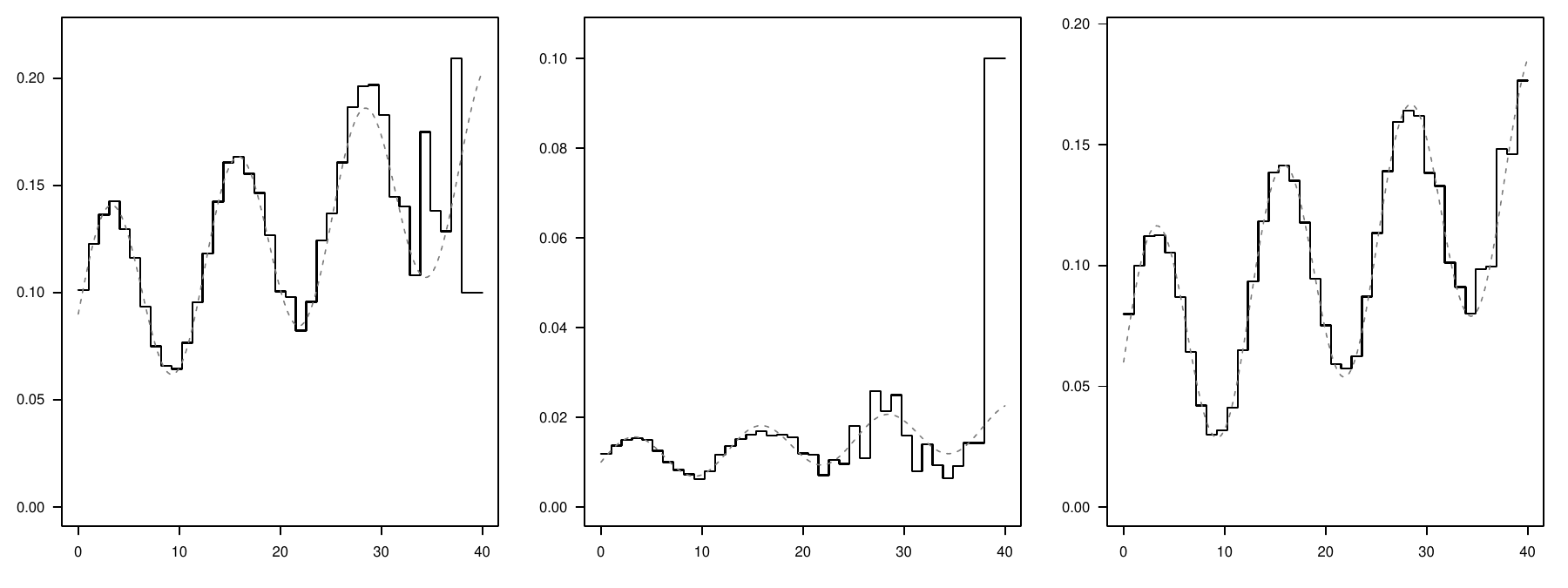}
\caption{Single-sample Markov illustration. From left to right, the panels show $\mu_{12}$, $\mu_{13}$ and $\mu_{23}$ as functions of calendar time. The solid step curves are the occurrence-exposure Poisson estimates, and the grey dashed curves are the true intensities.}
\label{fig:single_sample}
\end{figure}

\subsection{Bias-variance tradeoff}

We next examine the conditions of the pointwise central limit theorem in a setting where the sample size is fixed and the bin width is varied over a range. We estimate $\mu_{12}(t_0)$ at the interior time point $t_0 = 20$ using regular partitions of $[0,40]$ into $M$ bins, with $M$ ranging from $5$ to $80$. Thus $\Delta_n = 40/M$. For each value of~$M$, we simulate $R = 1000$ independent samples of size $n = 500$ under the model above and retain the occurrence-exposure estimate from the bin containing $t_0$.

Figure~\ref{fig:biasvariance} displays two quantities as functions of the bin width: the variance of the normalized estimation error
\begin{align*}
Z_n &= \sqrt{n\Delta_n}\bigl(\hat{\mu}_{12}(t_0) - \mu_{12}(t_0)\bigr),
\end{align*}
and the absolute bias scaled by $\sqrt{n\Delta_n}$. The latter is the bias on the scale relevant for the limit theorem.
\begin{figure}[t]
\centering
\includegraphics[width=0.5\textwidth]{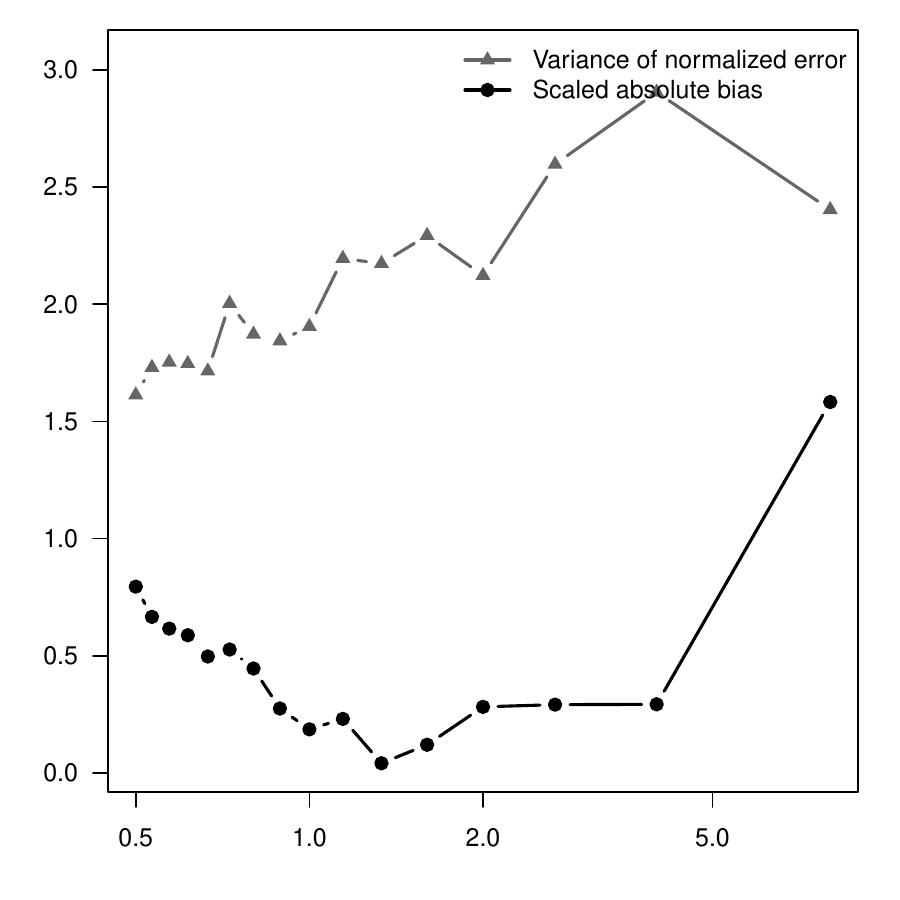}
\caption{Bias--variance tradeoff for pointwise estimation of $\mu_{12}(20)$. The horizontal axis is the bin width $\Delta_n$. The two curves show the variance of $Z_n = \sqrt{n\Delta_n}\bigl(\hat{\mu}_{12}(20)-\mu_{12}(20)\bigr)$ and the scaled absolute bias $\sqrt{n\Delta_n}\, |\BE[\hat{\mu}_{12}(20)]-\mu_{12}(20)|$.}
\label{fig:biasvariance}
\end{figure}
The figure shows the expected tradeoff. For large values of $\Delta_n$, the local approximation is poor and the scaled bias is non-negligible. When $\Delta_n$ becomes very small, the effective sample size within a bin deteriorates and the variance of $Z_n$ increases. The intermediate range is the one relevant for the asymptotic normality result.

\subsection{Asymptotic normality approximation}

To examine the distributional approximation more directly, we consider the empirical distribution of
\begin{align*}
Z_n &= \sqrt{n\Delta_n}\bigl(\hat{\mu}_{12}(t_0) - \mu_{12}(t_0)\bigr)
\end{align*}
under three representative mesh widths.

Figure~\ref{fig:clt} shows kernel density estimates of $Z_n$ based on the same $R = 1000$ replications for $M = 5$, $15$ and $75$. In the coarse regime, the distribution is visibly shifted, reflecting non-negligible discretization bias. In the very fine regime, it is too dispersed, which is consistent with insufficient effective sample size within a bin. In the intermediate regime, the Gaussian approximation is better.

\begin{figure}[t]
\centering
\includegraphics[width=\textwidth]{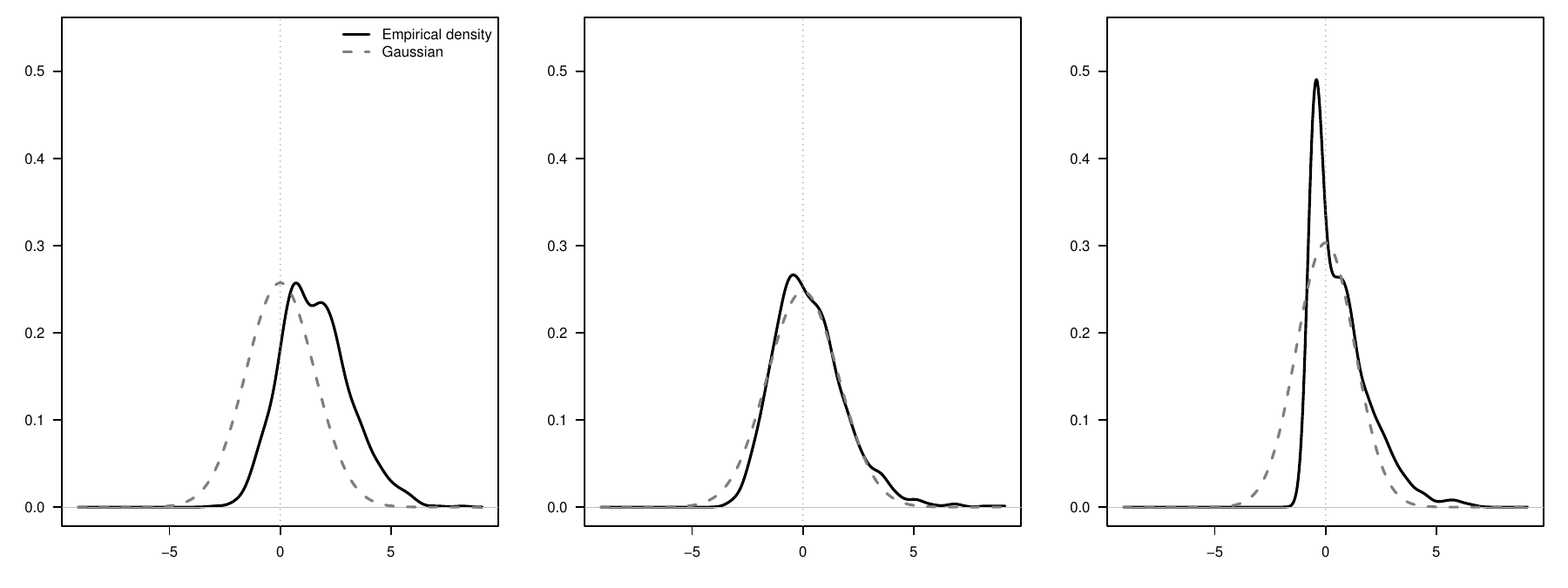}
\caption{Empirical densities of $Z_n = \sqrt{n\Delta_n}\bigl(\hat{\mu}_{12}(20)-\mu_{12}(20)\bigr)$ for three mesh widths. From left to right, the panels correspond to $M = 5$, $15$ and $75$. The dashed curves are centered Gaussian densities with variance matched to the corresponding Monte Carlo sample.}
\label{fig:clt}
\end{figure}
These Markov experiments illustrate that the central limit theorem is genuinely local, but also somewhat sensitive to the bias-variance tradeoff.
\subsection{Semi-Markov illustration}

We finally consider a semi-Markov model on the same state space $\CZ = \{1,2,3\}$, again with $Z_0 = 1$, absorbing state~$3$, and independent censoring $R \sim \mathrm{Unif}(10,40)$. We simulate $n = 100{,}000$ trajectories. The transition intensities are
\begin{align*}
\mu_{12}(t,u) &= 0.09 + 0.0018 t, \\
\mu_{13}(t,u) &= 0.01 + 0.0002 t, \\
\mu_{23}(t,u) &= 0.09 + 0.001\, t(1 + 0.1u) + \frac{0.2}{1 + \exp\{0.5(u-4)\}}.
\end{align*}
Thus the transitions out of state~$1$ depend only on calendar time, whereas the transition from state~$2$ to state~$3$ depends on both calendar time and duration. The semi-Markov occurrence-exposure estimator is fitted on a regular grid with mesh size $2$ in both coordinates, so that the estimate is piecewise constant on boxes of the form $[t_{m_1-1}, t_{m_1}) \times [u_{m_2-1}, u_{m_2})$.

Figure~\ref{fig:poisson_semi_3d_numeric} compares the fitted values with the true intensity surfaces on the admissible region $\{(t,u) : 0 \leq u \leq t\}$. Since three-dimensional displays are difficult to read in detail, Figure~\ref{fig:poisson_semi_slice_numeric} also reports diagonal sections along lines of the form $t-u=d$. For $\mu_{12}$ and $\mu_{13}$, only $d=0$ is relevant, because before the first jump the duration coincides with calendar time. For $\mu_{23}$ we show the sections $d \in \{1,5,10,20\}$ in order to display the duration effect more clearly. The estimated surfaces recover both the overall time trend and the higher transition intensity at short durations.

\begin{figure}[!htbp]
\centering
\includegraphics[width=0.45\textwidth, trim= 0in 0in 0in 0in,clip]{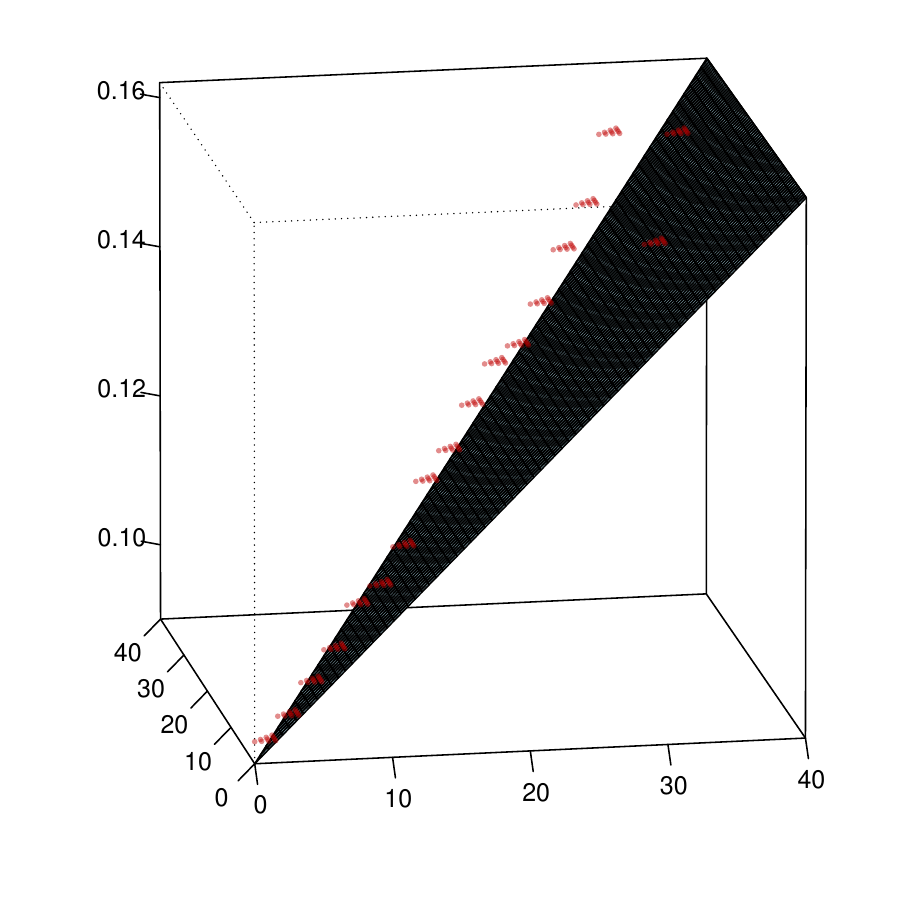}
\includegraphics[width=0.45\textwidth, trim= 0in 0in 0in 0in,clip]{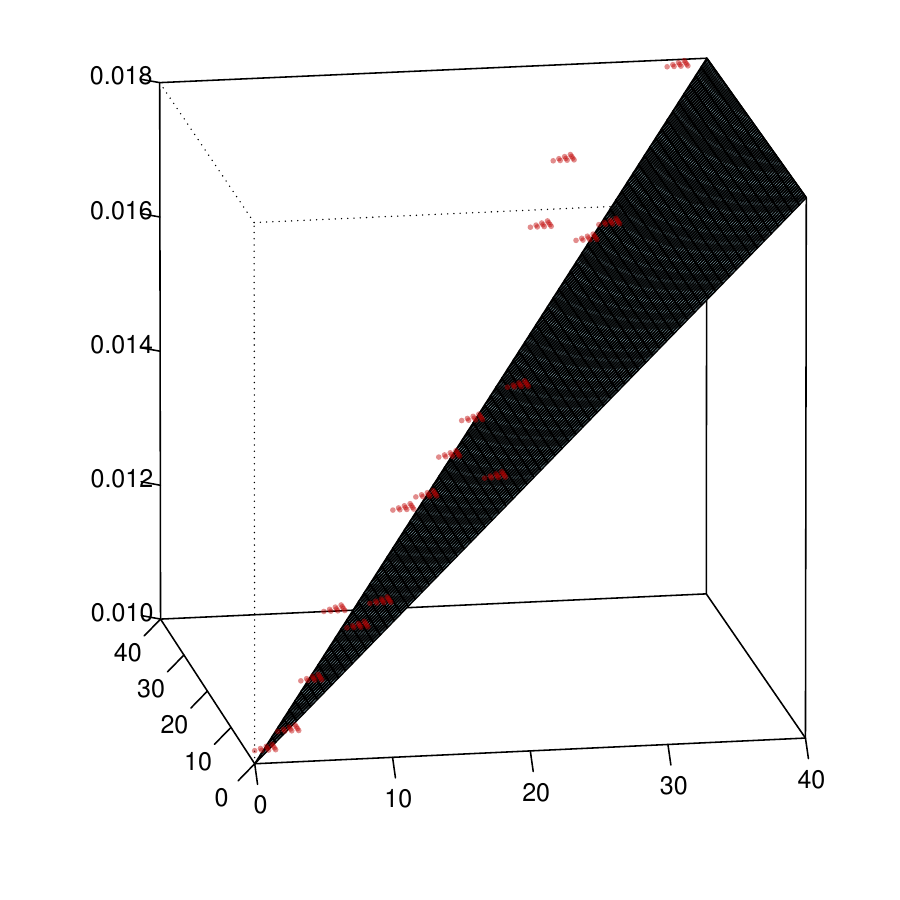}\\[0.05cm]
\includegraphics[width=0.45\textwidth, trim= 0in 0in 0in 0in,clip]{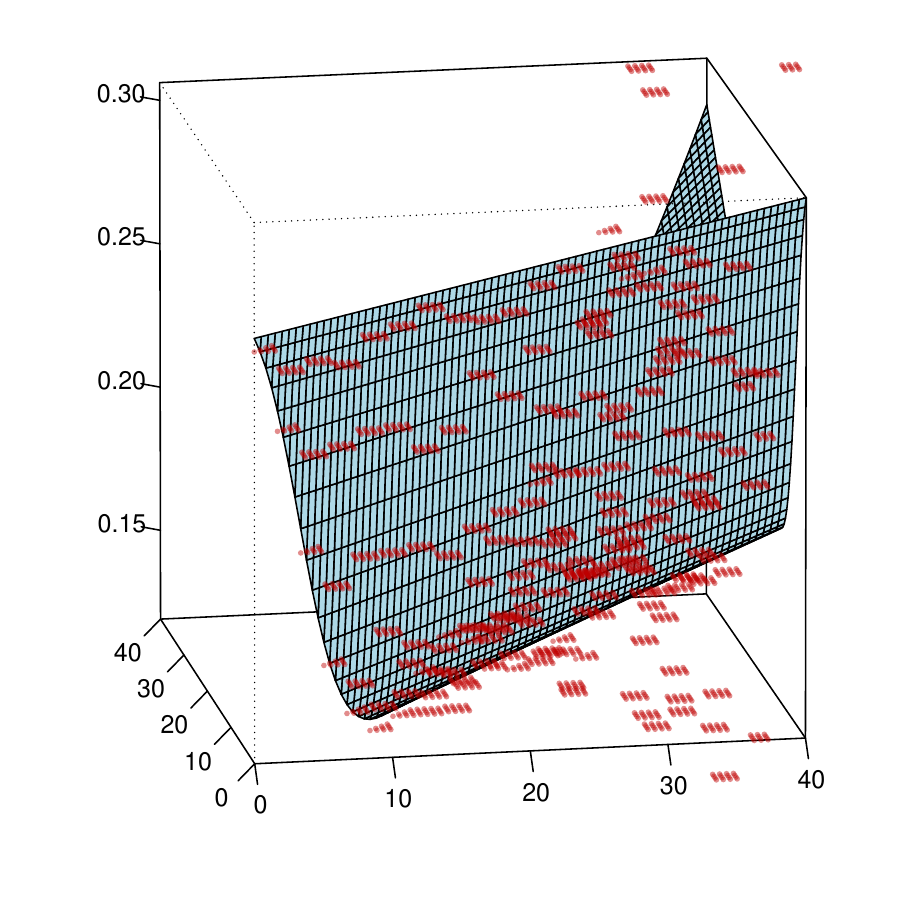}
\includegraphics[width=0.45\textwidth, trim= 0in 0in 0in 0in,clip]{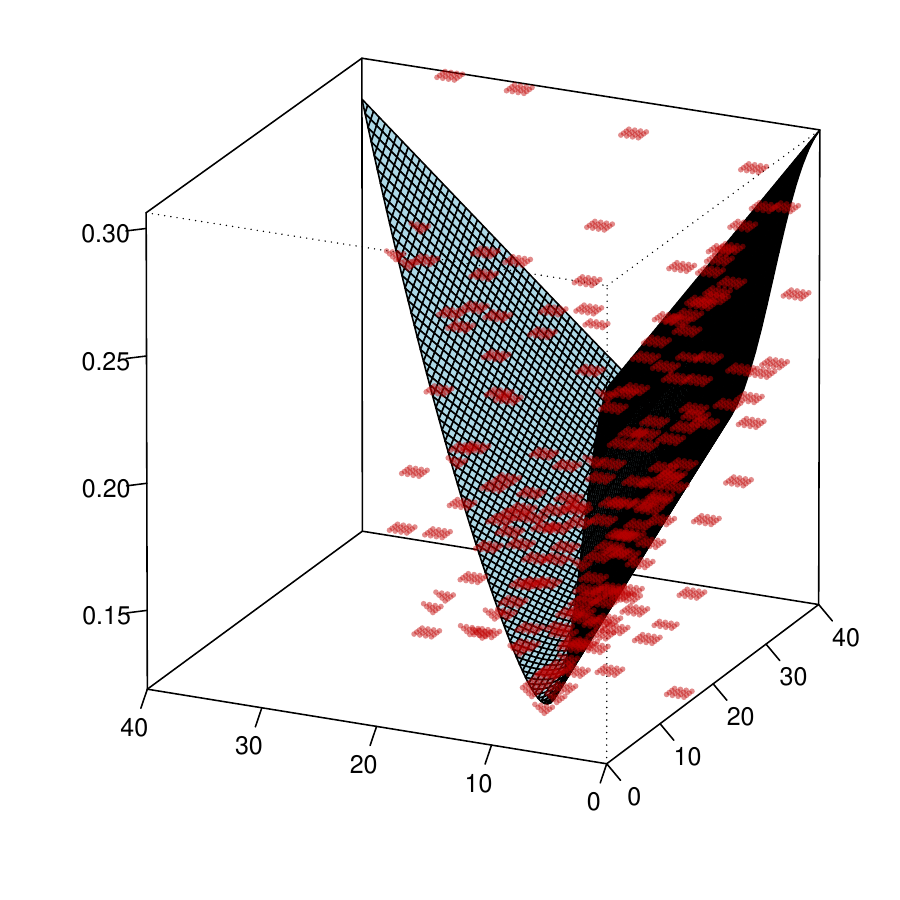}
\caption{Semi-Markov Poisson-regression fits on $(t,u)$. The top row shows $\mu_{12}$ and $\mu_{13}$; the bottom row shows $\mu_{23}$ from two viewing angles. The blue surface is the true intensity and the red points are the fitted occurrence-exposure values on the admissible region $0 \leq u \leq t$.}
\label{fig:poisson_semi_3d_numeric}
\end{figure}

\begin{figure}[!htbp]
\centering
\includegraphics[width=0.4\textwidth, trim= 0.3in 0.6in 0.3in 0.6in,clip]{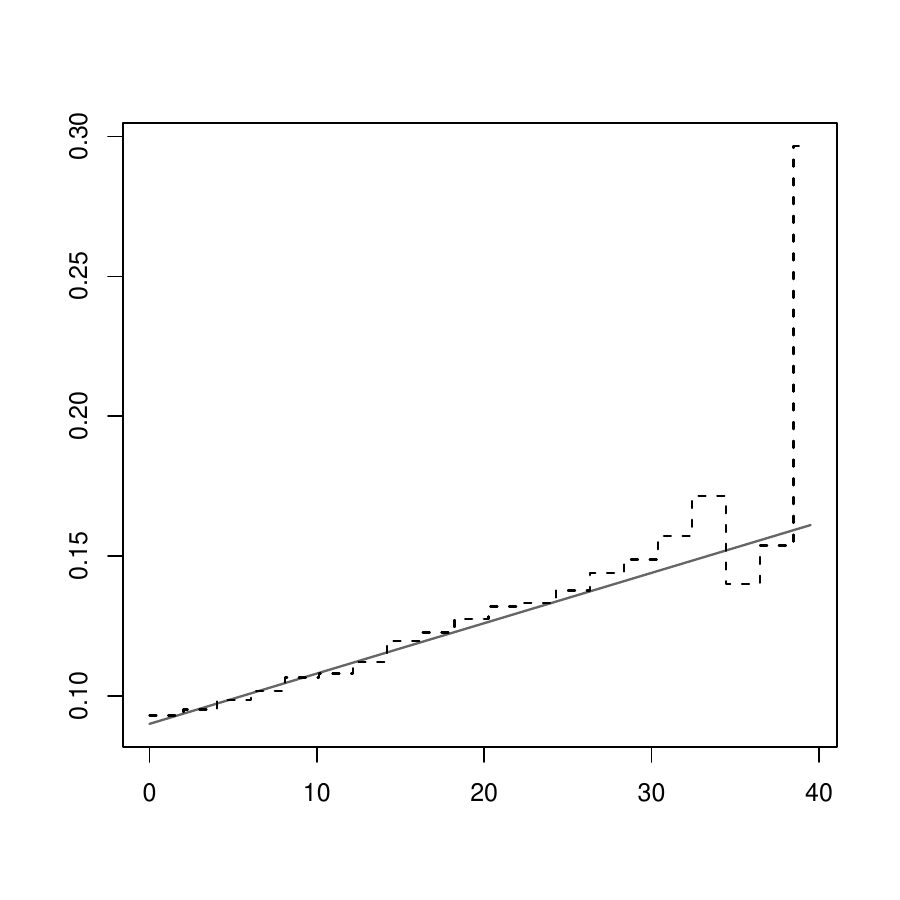}
\includegraphics[width=0.4\textwidth, trim= 0.3in 0.6in 0.3in 0.6in,clip]{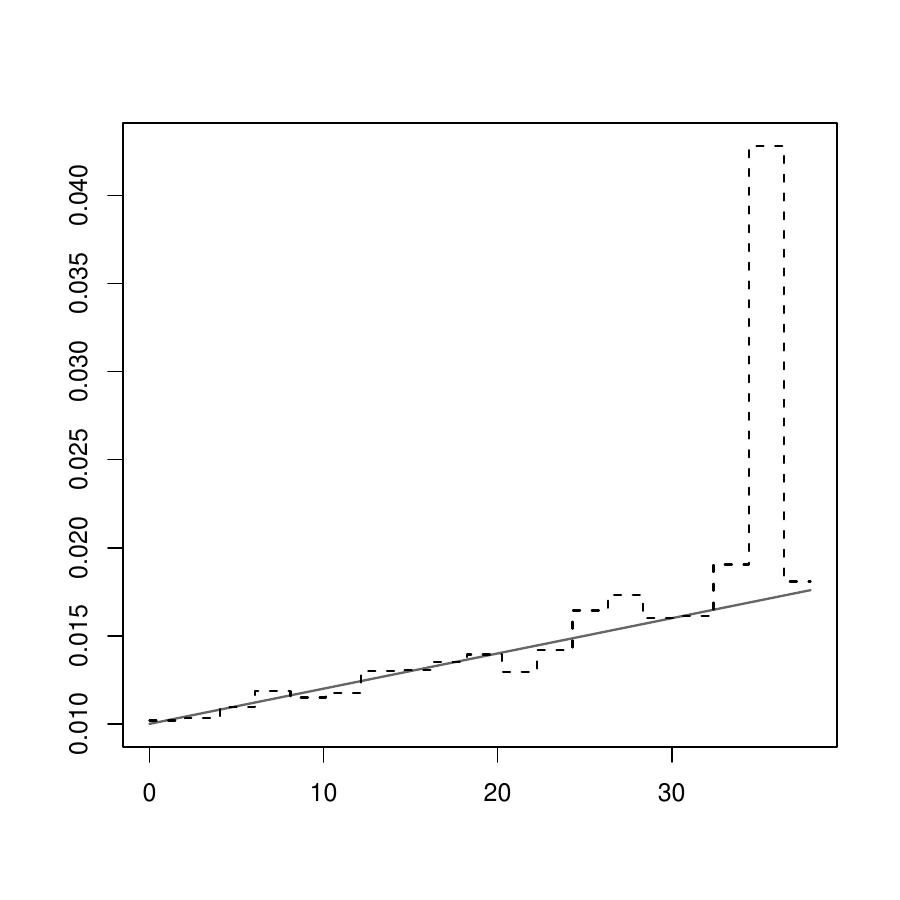}\\[0.05cm]
\includegraphics[width=0.8\textwidth, trim= 0.3in 0.6in 0.3in 0.6in,clip]{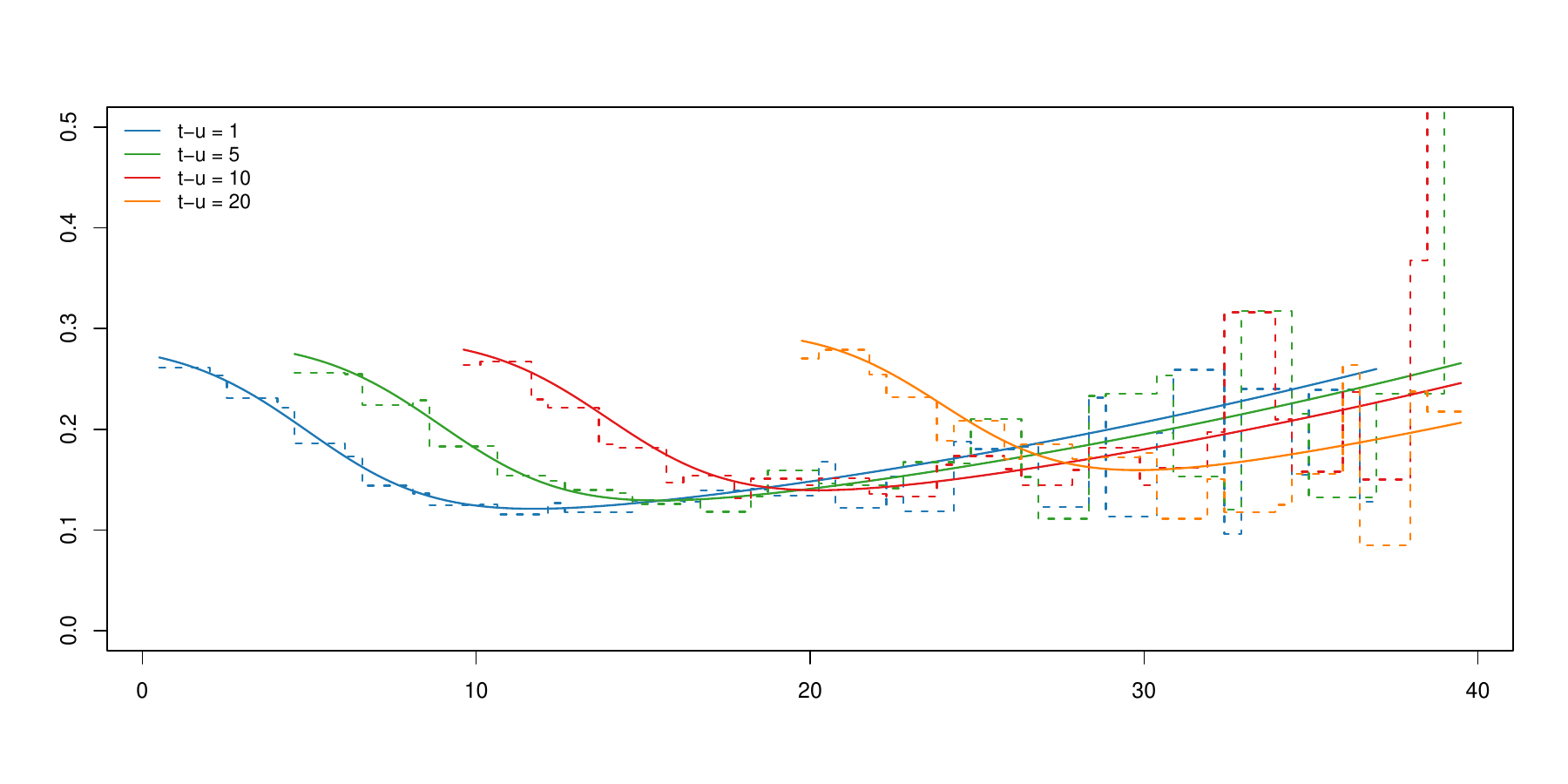}
\caption{Diagonal sections of the semi-Markov intensities along lines $t-u=d$. The top row corresponds to $\mu_{12}$ and $\mu_{13}$ with $d=0$. The lower panel corresponds to $\mu_{23}$ with $d \in \{1,5,10,20\}$. Solid curves denote the true intensities, and dashed step curves denote the fitted occurrence-exposure values.}
\label{fig:poisson_semi_slice_numeric}
\end{figure}

\section{Data application} \label{sec:Application}

\subsection{Data description and multi-state framework}
We illustrate the methods on the \texttt{LEC-DK19} (Loss of Earning Capacity -- Denmark 2019) dataset. The observation window runs from 31/01/2015 to 01/09/2019, the latter date serving as the analysis time $\eta$. Individuals are followed in a three-state model with states active ($1$), disabled ($2$), and reactivated ($3$).

The data contain $1{,}773$ disability events, corresponding to an incidence of $0.43\%$, and $1{,}133$ reactivation events, corresponding to $37.63\%$ of the disabled population. Following~\cite{BuchardtFurrerSandqvist}, we focus on the transition intensities $\mu_{12}$ and $\mu_{23}$. Since the number of deaths from the disabled state is small, the corresponding mortality intensity is set equal to zero in the present analysis.

\subsection{Standard and regularized Poisson estimation}
We discretize time using a regular grid $\mathcal{G} = \{t_0, t_1, \dots, t_M\}$ with bin width $\Delta_n = 0.1$ years. For each bin, the data are summarized by occurrence and exposure, and this yields the same Poisson likelihood structure as in the theoretical development. Duration is not present, so we restrict ourselves to a Markov specification. As a baseline, we use the grid-based occurrence-exposure estimator
\begin{align*}
\hat{\mu}_{jk}^{(m)} &= \frac{O_{jk}(m)}{E_j(m)}.
\end{align*}
At interior time points, Section~\ref{sec:main} yields
\begin{equation}
\sqrt{n\Delta_n}(\hat{\mu}_{jk}(t) - \mu_{jk}(t)) \xrightarrow{d} N\left(0, \frac{\mu_{jk}(t)}{p_j^c(t)}\right),
\end{equation}
where $p_j^c(t)$ is the occupation probability under censoring. This provides the benchmark intervals reported in Figure~\ref{fig:StandardPoisson}.
\begin{figure}[htbp!]
    \centering
    \includegraphics[width=0.4\textwidth]{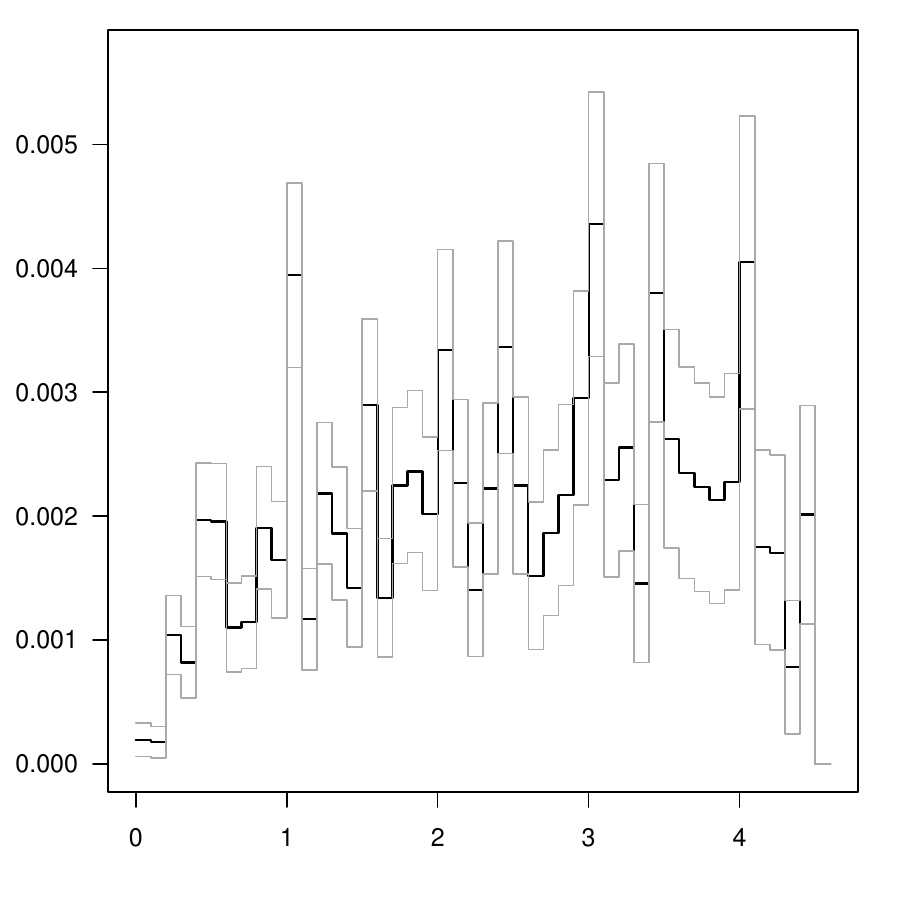}
    \includegraphics[width=0.4\textwidth]{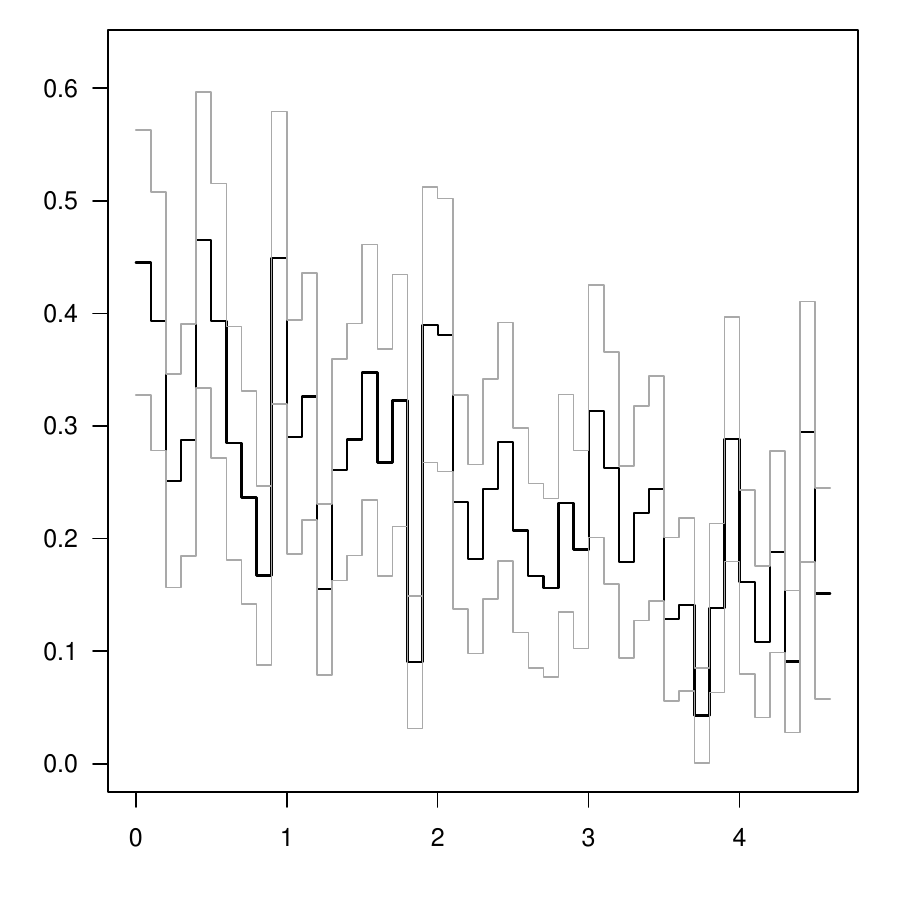}
    \caption{Standard occurrence-exposure estimates on the uniform $0.1$-year grid. The left panel shows $\hat{\mu}_{12}$ and the right panel shows $\hat{\mu}_{23}$, both as functions of time in years. Grey lines indicate pointwise $95\%$ intervals based on the asymptotic normal approximation from Section~\ref{sec:main}.}
    \label{fig:StandardPoisson}
\end{figure}

We now also consider a Poisson regression tree. The method recursively partitions the time axis by reductions in Poisson deviance and therefore replaces the fixed grid by a data-driven one. Each leaf corresponds to a pooled occurrence-exposure estimate on a larger time interval. The intervals shown in Figure~\ref{fig:CartPoisson} are obtained by applying the same variance formula to the pooled occurrence and exposure within the relevant leaf.

Finally, we fit a fused LASSO on the log-intensity parameters $\alpha_m = \log(\mu_m)$, defined by
\begin{equation}
\hat{\boldsymbol{\alpha}} = \underset{\boldsymbol{\alpha}}{\text{argmin}} \sum_{m=1}^M \left( E_m e^{\alpha_m} - O_m \alpha_m \right) + \lambda \sum_{m=2}^M |\alpha_m - \alpha_{m-1}|.
\end{equation}
The penalty encourages neighbouring bins to share a common level and thus produces a piecewise constant estimate with fewer jumps. We report heuristic bands in Figure~\ref{fig:LassoPoisson} by plugging the regularized estimate into the local Poisson variance formula.

Before turning to the plots, we note that the confidence bands should not be interpreted uniformly across the three methods. For the standard occurrence-exposure estimator, the pointwise bands are supported by the asymptotic theory above. For the tree-based and fused LASSO estimators, they should be regarded as heuristic.
\begin{figure}[htbp!]
    \centering
    \includegraphics[width=0.4\textwidth]{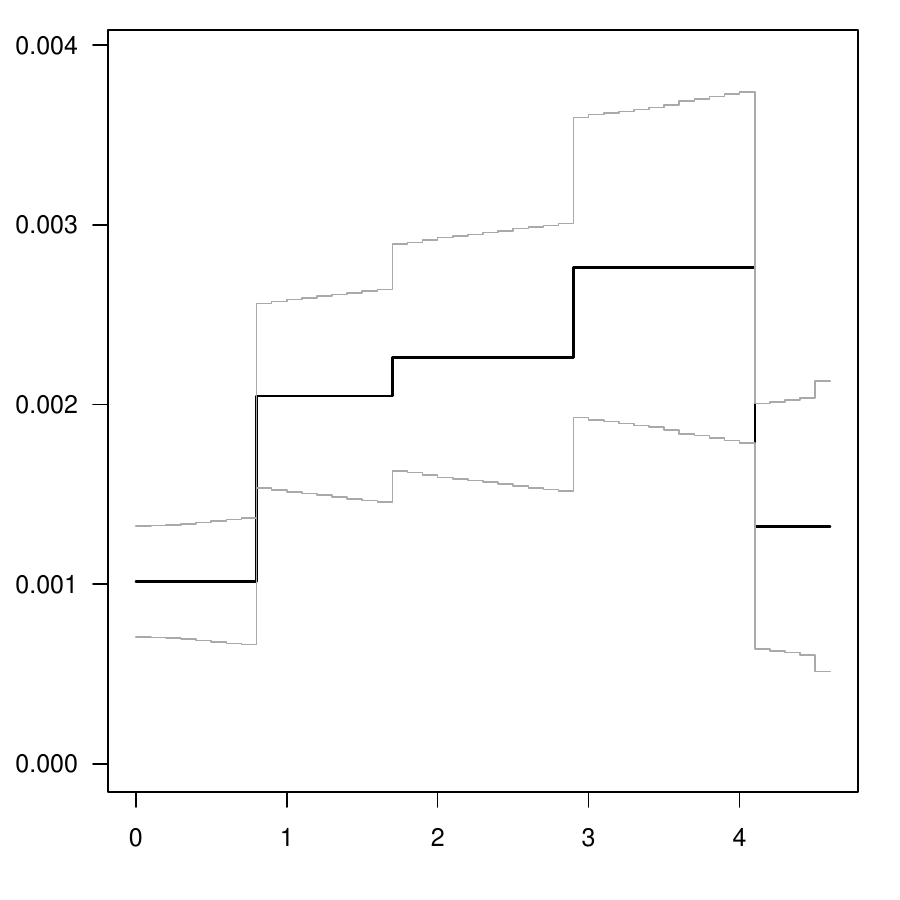}
    \includegraphics[width=0.4\textwidth]{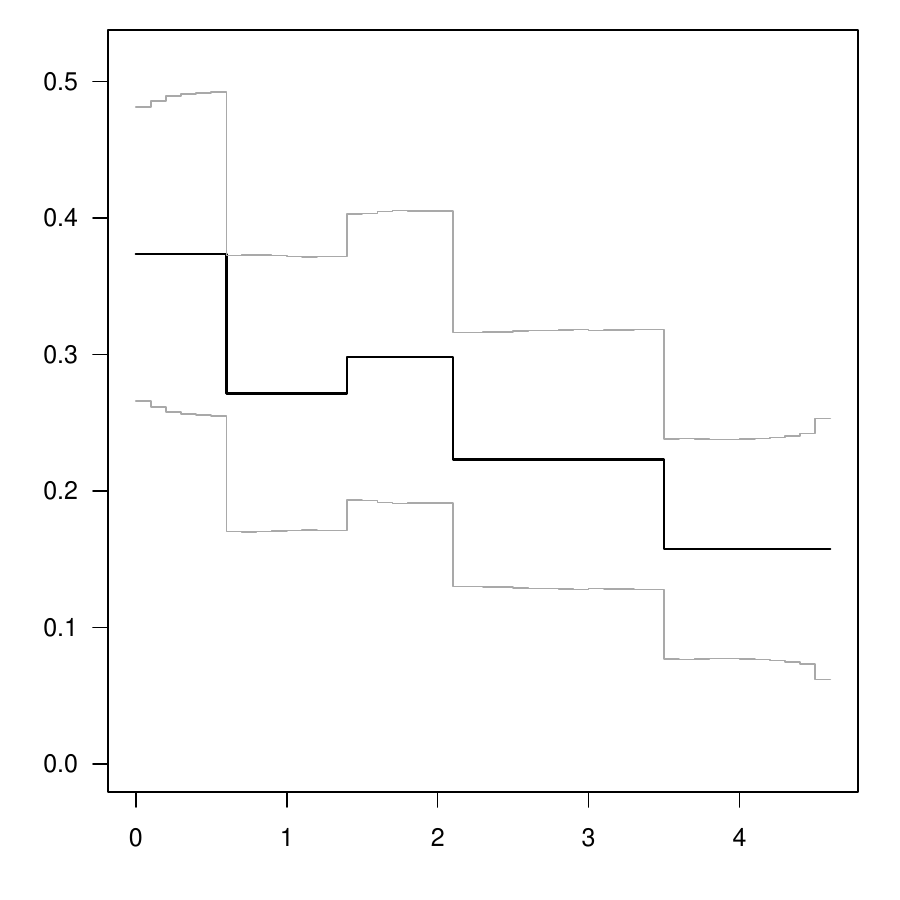}
    \caption{Poisson regression tree estimates. The left panel shows $\mu_{12}$ and the right panel shows $\mu_{23}$, both as functions of time in years. Grey lines indicate pointwise intervals obtained from the pooled occurrence and exposure within the leaf containing the time point.}
    \label{fig:CartPoisson}
\end{figure}

Figures~\ref{fig:StandardPoisson}, \ref{fig:CartPoisson}, and~\ref{fig:LassoPoisson} show the estimated disability and reactivation intensities. The standard occurrence-exposure estimator varies appreciably from bin to bin, especially for $\mu_{23}$. The tree-based method pools neighbouring bins into a small number of regimes, while the fused LASSO yields the most regular profile. The three approaches nevertheless lead to a similar qualitative picture of the main temporal variation in the two intensities. The tree-based and fused LASSO fits are included as data-adaptive comparisons, but such adaptive bin selection is not covered by our asymptotic theory in its present form. An extension in that direction, possibly with the same limit variance under additional assumptions, is left for future work.

\begin{figure}[htbp!]
    \centering
    \includegraphics[width=0.4\textwidth]{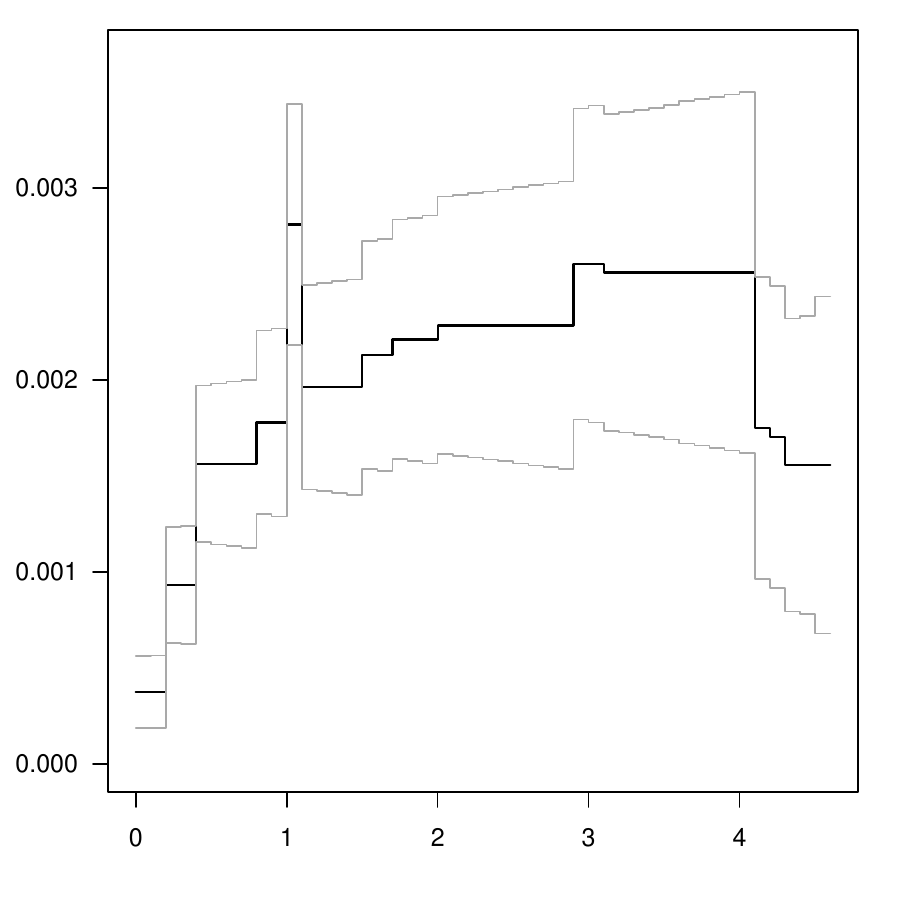}
    \includegraphics[width=0.4\textwidth]{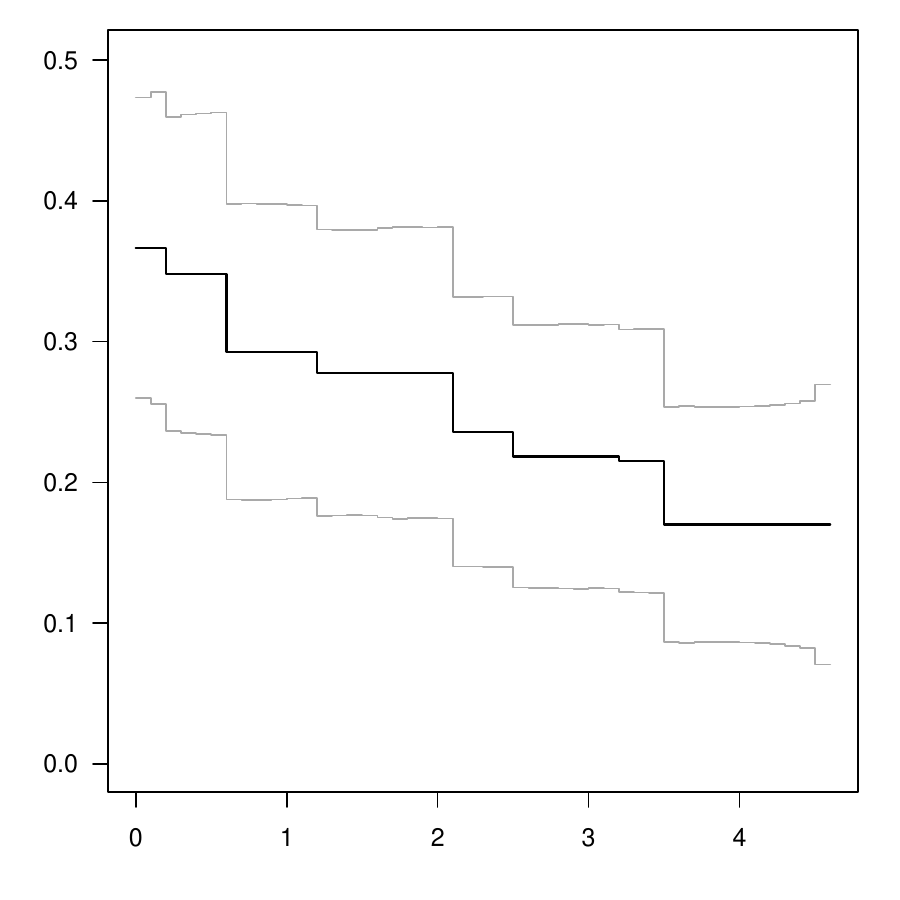}
    \caption{Fused LASSO estimates. The left panel shows $\mu_{12}$ and the right panel shows $\mu_{23}$, both as functions of time in years. The black curves are the regularized estimates, and the grey lines are heuristic pointwise bands obtained from the local Poisson variance formula.}
    \label{fig:LassoPoisson}
\end{figure}

Finally, we close by mentioning that piecewise constant estimators are only one possibility. If a smoother representation is desired, one may instead consider generalized additive models, parametric intensity models, or boosting methods based on Poisson loss. For such procedures, however, a pointwise asymptotic normality theory for uncertainty quantification is currently an open problem.

\section{Discussion and further work}\label{sec:Discussion}

In this article, we have obtained new results for Poisson regression for Markov and semi-Markov jump processes where the time and duration intervals are shrinking as data grows. We have presented and proved asymptotic normality results for the occurence/exposure rate estimators under appropriate shrinking rates on the interval lengths and without structural assumptions on the true intensities. 

We see many avenues for further studies. First and foremost, the partitions were deterministic and even equidistant, resulting in possibly suboptimal splits for a finite data sample. It is of interest to pursue data-driven splits and investigate whether the same or suitably adapted asymptotic normality results still hold. For practical implementations, one should also develop different test statistics to determine the optimal splits for concrete data. In the present deterministic partitioning scheme, it is relevant to investigate whether other asymptotic properties hold such as uniform consistency. Uniform consistency of the transition rates in our setup would, via continuity of the product integral in the supremum norm (see \cite{GillJohansen}), translate into uniform consistency of transition probabilities, which is of interest in especially life insurance applications such as computing cash flows. Lastly, as was already mentioned above, it is of interest to pursue asymptotic normality for the tree-based and fused LASSO estimators used in the numerical study, allowing for theoretically sound confidence bands.


%

\begin{funding}
Both authors were supported by the Carlsberg Foundation, grant CF23-1096, and by the project framework InterAct.
\end{funding}



\newpage
\bibliographystyle{imsart-number} 
\bibliography{main.bib}       


\newpage

\appendix

\section{Additional proofs}\label{sec:Proofs}

\begin{proof}[Proof of Lemma~\ref{lem:meanvarianceMarkov}]
We prove each assertion.
\begin{enumerate}
\item Note that
\begin{equation*}
    X_{i, n} = \int_{I_n} 1_{(s \leq R^i)} N_{jk}^i(\dd s)
\end{equation*}
and thus by random right-censoring,
\begin{align*}
    \mu_n^{jk} &= \int_{I_n} \BE[1_{(s < R)}1_{(Z_{s-} = j)}] q_{jk}(\dd s) = \int_{I_n} \BP(s \leq R, Z_{s-} = j) \mu_{jk}(s)\dd s \\
    &= \int_{I_n} p_j^\texttt{c}(s)\mu_{jk}(s)\dd s.
\end{align*}
The asymptotic expression follows by applying a Taylor expansion around $t$.
\item We have
\begin{equation*}
    \eta_n^j = \int_{I_n} \BE[1_{(s < R)}1_{(Z_{s-} = j)}] \dd s = \int_{I_n} p_j^\texttt{c}(s)\dd s = \Delta_n p_j^\texttt{c}(t) + O(\Delta_n^2),
\end{equation*}
again by a Taylor expansion.
\item Start by noting that
\begin{equation*}
    X_{i,n} = N_{jk}^i(t_m \land R^i) - N_{jk}^i(t_{m - 1} \land R^i).
\end{equation*}
Defining $M_{jk}(t) := N_{jk}(t) - \int_0^t 1_{(Z_{s-} = j)}\mu_{jk}(s)\dd s$, we have
\begin{align*}
    \BE[X_{i,n}^2] &= \BE[(M_{jk}(t_m \land R) - M_{jk}(t_{m - 1} \land R))^2] \\
    &+ \BE\Big[\Big(\int_{I_n} 1_{(s \leq R)} 1_{(Z_{s-} = j)}\mu_{jk}(s)\dd s\Big)^2 \Big] \\
    &+ 2\BE\Big[(M_{jk}(t_m \land R) - M_{jk}(t_{m - 1} \land R)) \int_{I_n} 1_{(s \leq R)} 1_{(Z_{s-} = j)}\mu_{jk}(s)\dd s \Big].
\end{align*}
Since $\mu_{jk}$ is $C^1$ close to $t$, for large enough $n$, $\mu_{jk}$ is continuous on $I_n$ and thus bounded. From this observation, it is not difficult to see that the second and third terms are $O(\Delta_n^2)$. Using Lemma~\ref{lem:martingaleshiftedsecondmoment}, we get for the first term that
\begin{equation*}
    \BE[(M_{jk}(t_m \land R) - M_{jk}(t_{m - 1} \land R))^2] = \BE\Big[\int_{I_n} 1_{(s \leq R)} 1_{(Z_{s-} = j)} \mu_{jk}(s)\dd s \Big]
\end{equation*}
which equals $\mu_n^{jk}$. Hence
\begin{equation*}
    \operatorname{Var}(X_{i, n}) = \mu_n^{jk} - (\mu_n^{jk})^2 + O(\Delta_n^2)
\end{equation*}
and from 1, it immediately follows that
\begin{equation*}
	\operatorname{Var}(X_{1, n}) = \Delta_n( p_j^\texttt{c}(t)\mu_{jk}(t) + O(\Delta_n)). 
\end{equation*}
\item Finally,
\begin{align*}
    \operatorname{Var}(Y_{1, n}) &= \BE[Y_{1, n}^2] - (\eta_n^j)^2 = \BE\Big[\int_{I_n} \int_{I_n} 1_{(s \lor u < R)}1_{(Z_s = j, Z_u = j)}\dd s \dd u \Big] - (\eta_n^j)^2 \\
    &= \int_{I_n}\int_{I_n} \BP(s \lor u < R, Z_s = j, Z_u = j)\dd s \dd u - (\eta_n^j)^2.
\end{align*}
Applying a two-dimensional Taylor expansion in the point $(t, t)$ yields
\begin{equation*}
    \operatorname{Var}(Y_{1, n}) = \Delta_n^2 p_j^\texttt{c}(t) + O(\Delta_n^2)
\end{equation*}
as claimed.
\end{enumerate}
\end{proof}

\begin{proof}[Proof of Lemma~\ref{lem:Lyapunov}]
Letting $M_{jk}^i$ denote the counting process martingale for $N_{jk}^i$, we can decompose
\begin{equation*}
    X_{i, n} - \mu_n^{jk} = \int_{I_n} 1_{(s \leq R^i)}\dd M_{jk}^i(s) + \int_{I_n} (1_{(s \leq R^i)}1_{(Z_{s-} = j)} - p_j^\texttt{c}(s))\mu_{jk}(s)\dd s,
\end{equation*}
and so by Minkowski's inequality,
\begin{align*}
    \BE[(X_{i, n} - \mu_n^{jk})^4]^{1/4} &\leq \BE\Big[ \Big(\int_{I_n} 1_{(s \leq R^i)}\dd M_{jk}^i(s) \Big)^4\Big]^{1/4} \\
    &\quad+ \BE\Big[\Big(\int_{I_n} (1_{(s \leq R)}1_{(Z_{s-} = j)} - p_j^\texttt{c}(s))\mu_{jk}(s)\dd s \Big)^4 \Big]^{1/4}
\end{align*}
Using the formula for the fourth moment in Corollary~\ref{cor:shiftedfourthmoment}, we see that the first term is $O(\Delta_n^{1/4})$ by using that we have assumed fourth moments and the fact that the intensities are bounded on $I_n$ for large enough $n$. The integrand in the second term is bounded (for $n$ large enough), and so we can conclude that
\begin{equation*}
    \BE[(X_{i, n} - \mu_n^{jk})^4] = O(\Delta_n). 
\end{equation*}

Now, the fourth moment of each term in the sum
\begin{equation*}
    \sum_{i = 1}^n \frac{X_{i, n} - \mu_n^{jk}}{\sqrt{\operatorname{Var}(\sum_{i = 1}^n X_{i,n}})}
\end{equation*}
is by Lemma~\ref{lem:meanvarianceMarkov} given by
\begin{equation*}
    \BE\Big[\Big(\frac{X_{i, n} - \mu_n^{jk}}{\sqrt{\operatorname{Var}(\sum_{i = 1}^n X_{i,n}})}\Big)^4\Big] = \frac{O(\Delta_n)}{(\operatorname{Var}(\sum_{i = 1}^n X_{i, n}))^2} = O\Big(\frac{\Delta_n}{n^2 \Delta_n^2} \Big) = O\Big(\frac{1}{n} \cdot \frac{1}{n\Delta_n} \Big),
\end{equation*}
so
\begin{equation*}
    \sum_{i = 1}^n \BE\Big[\Big(\frac{X_{i, n} - \mu_n^{jk}}{\sqrt{\operatorname{Var}(\sum_{i = 1}^n X_{i,n}})}\Big)^4\Big] = O\Big(\frac{1}{n\Delta_n}\Big) \to 0,
\end{equation*}
showing that Lyapunov's condition of order 4 holds.
\end{proof}

\begin{proof}[Proof of Corollary \ref{cor:asymptoticIndependenceMarkov}]
To obtain joint asymptotic dependence, we apply the Cramér--Wold device (see e.g. Corollary 6.5 in \cite{Kallenberg}). Let $\lambda_1, \lambda_2 \in \BR$. Then we need to show weak convergence of
\begin{align*}
    &\lambda_1 \sqrt{n\Delta_n}(\hat{\mu}^{m(t)} - \mu_{jk}(t)) + \lambda_2 \sqrt{n\Delta_n}(\hat{\mu}^{m(s)} - \mu_{jk}(s)) = \\
    &\lambda_1 \sqrt{n\Delta_n} \frac{O_{jk}(m(t)) - n\mu_n^{jk}(t)}{E_j(m(t))} + \lambda_2 \sqrt{n\Delta_n} \frac{O_{jk}(m(s)) - n\mu_n^{jk}(s)}{E_j(m(s))} + \\
    &\lambda_1 \sqrt{n\Delta_n} \frac{n\mu_n^{jk}(t) - \mu_{jk}(t)E_j(m(t))}{E_j(m(t))} + \lambda_2 \sqrt{n\Delta_n} \frac{n\mu_n^{jk}(s) - \mu_{jk}(s)E_j(m(s))}{E_j(m(s))}.
\end{align*}
From the proof of Theorem \ref{thm:PoissonMarkov}, we know that the latter two terms converge to zero in probability, so we need only show weak convergence of the sum of the first two terms. Recall that $E_j(m(t)) = n \Delta_n (p_j^\texttt{c}(t) + o_\BP(1))$, so we need to establish weak convergence of
\begin{equation*}
    \lambda_1 \frac{O_{jk}(m(t)) - n\mu_n^{jk}(t)}{\sqrt{n\Delta_n}(p_j^\texttt{c}(t) + o_\BP(1))} + \lambda_2 \frac{O_{jk}(m(s)) - n\mu_n^{jk}(s)}{\sqrt{n\Delta_n}(p_j^\texttt{c}(s) + o_\BP(1))}.
\end{equation*}
By subtracting this quantity without the $o_\BP(1)$ terms, it becomes clear that the difference is $o_\BP(1)$ so that we may ignore the $o_\BP(1)$ terms. Define the array
\begin{equation*}
    W_{ni}(t, s) := \lambda_1 \frac{X_{i, n}(t) - \mu_n^{jk}(t)}{\sqrt{n\Delta_n}p_j^\texttt{c}(t)} + \lambda_2 \frac{X_{i, n}(s) - \mu_n^{jk}(s)}{\sqrt{n\Delta_n}p_j^\texttt{c}(s)}.
\end{equation*}
The $W_{ni}$ have mean zero and are independent within rows. From Minkowski's inequality,
\begin{align*}
    \BE[W_{ni}^4(t, s)]^{1/4} &\leq \lambda_1 \frac{\BE[(X_{i, n}(t) - \nu_n^{jk}(t))^4]^{1/4}}{\sqrt{n\Delta_n} p_j^\texttt{c}(t)} + \lambda_2 \frac{\BE[(X_{i, n}(s) - \nu_n^{jk}(s))^4]^{1/4}}{\sqrt{n\Delta_n} p_j^\texttt{c}(s)} \\
    &= O\left(\frac{\Delta_n^{1/4}}{\sqrt{n\Delta_n}}\right) = O(n^{-1/2}\Delta_n^{-1/4})
\end{align*}
where we again used results from the proof of Theorem \ref{thm:PoissonMarkov}. It follows that
\begin{equation*}
    \sum_{i = 1}^n \BE[W_{ni}(t, s)^4] = O\left(\frac{1}{n\Delta_n}\right),
\end{equation*}
showing that the Lyapunov condition holds. We need only compute the limiting variance. From Lemma \ref{lem:meanvarianceMarkov},
\begin{align*}
    \Var(W_{ni}(t, s)) &= \lambda_1^2 \frac{\Var(X_{i, n}(t))}{n\Delta_n p_j^\texttt{c}(t)} + \lambda_2^2 \frac{\Var(X_{i, n}(s))}{n\Delta_n p_j^\texttt{c}(s)} + \frac{2\lambda_1\lambda_2}{n\Delta_n p_j^\texttt{c}(t)p_j^\texttt{c}(s)}\Cov(X_{i, n}(t), X_{i, n}(s)) \\
    &= \lambda_1^2 \frac{\mu_{jk}(t)}{p_j^\texttt{c}(t)} + \lambda_2^2 \frac{\mu_{jk}(s)}{p_j^\texttt{c}(s)} + O(\Delta_n) + \frac{2\lambda_1\lambda_2}{n\Delta_n p_j^\texttt{c}(t)p_j^\texttt{c}(s)}\Cov(X_{i, n}(t), X_{i, n}(s)).
\end{align*}
Now suppose $n$ is so large such that $I_n(t) \neq I_n(s)$. Then $I_n(t) \cap I_n(s) = \emptyset$, and we may without loss of generality assume $s < t$ such that $t_{m(s)} < t_{m(t) - 1}$. Decompose
\begin{equation*}
    X_{i, n}(t) = M_{jk}^i(t_{m(t)} \land R^i) - M_{jk}^i(t_{m(t) - 1} \land R^i) + \int_{I_n(t)} 1_{(v \leq R^i)} 1_{(Z_{v-}^i = j)} \mu_{jk}(v)\dd v
\end{equation*}
and similarly for $X_{i, n}(s)$. Then
\begin{align*}
    X_{i, n}(t)X_{i, n}(s) &= (M_{jk}^i(t_{m(t)} \land R^i) - M_{jk}^i(t_{m(t) - 1} \land R^i))(M_{jk}^i(t_{m(s)} \land R^i) - M_{jk}^i(t_{m(s) - 1} \land R^i)) \\
    &+ (M_{jk}^i(t_{m(t)} \land R^i) - M_{jk}^i(t_{m(t) - 1} \land R^i))\int_{I_n(s)} 1_{(v \leq R^i)} 1_{(Z_{v-}^i = j)} \mu_{jk}(v)\dd v \\
    &+ (M_{jk}^i(t_{m(s)} \land R^i) - M_{jk}^i(t_{m(s) - 1} \land R^i))\int_{I_n(t)} 1_{(v \leq R^i)} 1_{(Z_{v-}^i = j)} \mu_{jk}(v)\dd v \\
    &+ \int_{I_n(t)} 1_{(v \leq R^i)} 1_{(Z_{v-}^i = j)} \mu_{jk}(v)\dd v \int_{I_n(s)} 1_{(v \leq R^i)} 1_{(Z_{v-}^i = j)} \mu_{jk}(v)\dd v \\
    &=: (1) + (2) + (3) + (4).
\end{align*}
Taking expectation, we have, just like in the proof of Lemma \ref{lem:meanvarianceMarkov} that terms (2), (3) and (4) are $O(\Delta_n^2)$. As for the first term, we can apply the tower property with $\mathcal{G}_{t_{m(s)}}$ and the fact that 
\begin{equation*}
    \BE[M_{jk}(t_{m(s)} \land R) - M_{jk}(t_{m(s) - 1} \land R) \mid \mathcal{G}_{t_{m(s)}}] = 0
\end{equation*}
to conclude that $\BE[(1)] = 0$. All in all, we may conclude that $\BE[X_{i, n}(t)X_{i, n}(s)] = O(\Delta_n^2)$ and thus
\begin{align*}
    \sum_{i = 1}^n \Var(W_{i, n}(t, s)) &= \lambda_1^2 \frac{\mu_{jk}(t)}{p_j^\texttt{c}(t)} + \lambda_2^2 \frac{\mu_{jk}(s)}{p_j^\texttt{c}(s)} + O(\Delta_n) \\
    &\to \begin{pmatrix} \lambda_1 & \lambda_1 \end{pmatrix} \begin{pmatrix} \frac{\mu_{jk}(t)}{p_j^\texttt{c}(t)} & 0 \\ 0 & \frac{\mu_{jk}(s)}{p_j^\texttt{c}(s)} \end{pmatrix} \begin{pmatrix}
        \lambda_1 \\ \lambda_2 \end{pmatrix}.
\end{align*}
The proof is now complete. 
\end{proof}

\begin{proof}[Proof of Lemma \ref{lem:meanvariancesemiMarkov}]
The proof is similar to the one for Lemma \ref{lem:meanvarianceMarkov}. The key difference is in how the Taylor expansions are applied.
\begin{enumerate}
\item By random right-censoring, 
\begin{align*}
    \BE[X_{i, n}] &= \BE\Big[\int_{I_n^{(1)}} 1_{(s \leq R)}1_{(u_{m_2} - 1 \leq U_{s-} < u_{m_2})} 1_{(Z_{s-}^i = j)} \mu_{jk}(s, U_{s-})\dd s\Big] \\
    &= \int_{I_n^{(1)}} \BE[1_{(s \leq R)}1_{(u_{m_2} - 1 \leq U_{s-} < u_{m_2})} 1_{(Z_{s-} = j)} \mu_{jk}(s, U_{s-})] \dd s. 
\end{align*}
The joint measure of $(Z_t, U_t)$ is for $F \subseteq \mathcal{Z}$ and $G \in \mathcal{B}(\BR)$ given by
\begin{equation*}
    p_F(t, G) = \sum_{j \in F} \int_G p_j(t, \dd u),
\end{equation*}
where $p_j(t, u) = \BP(Z_t = j, U_t \leq u)$ so the inner expectation equals
\begin{equation*}
    \int_{I_n^{(2)}} \mu_{jk}(s, v) p_j^\texttt{c}(s, \dd v) = \int_{I_n^{(2)}} \mu_{jk}(s, v) \partial_2p_j^\texttt{c}(s, v)\dd v.
\end{equation*}
To summarise,
\begin{equation*}
    \mu_n^{jk} = \int_{I_n^{(1)}}\int_{I_n^{(2)}} \mu_{jk}(s, v) \partial_2p_j^\texttt{c}(s, v)\dd v \dd s.
\end{equation*}
Since both $p_j^\texttt{c}$ and $\mu_{jk}$ are $C^2$, the integrand is $C^1$ in the second variable and $C^2$ in the first. Do a Taylor expansion around $s = t$ to get
\begin{equation*}
    \mu_n^{jk} = \Delta_n^{(1)} \int_{I_n^{(2)}} \mu_{jk}(t, v) \partial_2 p_j^\texttt{c}(t, v)\dd v + O(\big(\Delta_n^{(1)}\big)^2 \Delta_n^{(2)}).
\end{equation*}
Next, do a Taylor expansion around $v = u$ to obtain
\begin{equation*}
    \mu_n^{jk} = \Delta_n^{(1)}\Delta_n^{(2)}\mu_{jk}(t, u)\partial_2p_j^\texttt{c}(t, u) + O\Big(\big(\Delta_n^{(1)}\big)^2 \Delta_n^{(2)}\Big) + O\Big(\Delta_n^{(1)} \big(\Delta_n^{(2)}\big)^2\Big).
\end{equation*}
\item We have
\begin{align*}
    \eta_n^j &= \BE\Big[\int_{I_n^{(1)}}1_{(u_{m_2 - 1} \leq U_{s-} < u_{m_2})}1_{(s \leq R)}1_{(Z_{s-} = j)}\dd s \Big] \\
    &= \int_{I_n^{(1)}}p_j^\texttt{c}(s, u_{m_2}) - p_j^\texttt{c}(s, u_{m_2 - 1})\dd s.
\end{align*}
Start by doing a second order Taylor expansion around $u$ in the second coordinate to get
\begin{equation*}
    p_j^\texttt{c}(s, v) = p_j^\texttt{c}(s, u) + \partial_2p_j^\texttt{c}(s, u)(v - u) + O\Big(\big(\Delta_n^{(2)}\big)^2\Big).
\end{equation*}
Due to continuity of the second derivative, we may choose the $O$-term to be independent of $s$ since the interval $I_n^{(1)}$ is bounded. Using this expansion on both terms in the integral yields
\begin{equation*}
    \eta_n^j = \Delta_n^{(2)}\int_{I_n^{(1)}} \partial_2p_j^\texttt{c}(s, u)\dd s + O\Big(\Delta_n^{(1)} \big(\Delta_n^{(2)}\big)^2\Big).
\end{equation*}
Finally, do a first order Taylor expansion around $s = t$ to obtain the desired,
\begin{equation*}
    \eta_n^j = \Delta_n^{(1)} \Delta_n^{(2)} \partial_2p_j^\texttt{c}(t, u) + O\Big(\big(\Delta_n^{(1)}\big)^2 \Delta_n^{(2)}\Big) + O\Big(\Delta_n^{(1)} \big(\Delta_n^{(2)}\big)^2\Big).
\end{equation*}
\item To compute $\operatorname{Var}(X_{i, n})$, we apply a similar strategy as for the Markov case. We decompose
\begin{align*}
    \BE[X_{i, n}^2] &= \BE\Big[\Big(\int_{I_n^{(1)}} 1_{(s \leq R)}1_{(u_{m_2 - 1} \leq U_{s-} < u_{m_2})} \dd N_{jk}(s) \Big)^2 \Big] \\
    &= \BE\Big[\Big(\int_{I_n^{(1)}} 1_{(s \leq R)}1_{(u_{m_2 - 1} \leq U_{s-} < u_{m_2})} \dd M_{jk}(s) \Big)^2 \Big] \\
    &+ \BE\Big[\Big(\int_{I_n^{(1)}}1_{(s \leq R)}1_{(u_{m_2 - 1} \leq U_{s-} < u_{m_2})} 1_{(Z_{s-} = j)}\mu_{jk}(s, U_{s-}) \dd s \Big)^2 \Big] \\
    &+2\BE\Big[\Big(\int_{I_n^{(1)}} 1_{(s \leq R)}1_{(u_{m_2 - 1} \leq U_{s-} < u_{m_2})} \dd M_{jk}(s) \Big) \\
    &\cdot \Big(\int_{I_n^{(1)}}1_{(s \leq R)}1_{(u_{m_2 - 1} \leq U_{s-} < u_{m_2})} 1_{(Z_{s-} = j)}\mu_{jk}(s, U_{s-})\dd s \Big)\Big] \\
    &=: (1) + (2) + (3).
\end{align*}
It follows immediately from Theorem II.3.1 in \cite{ABGK} that
\begin{equation*}
    (1) = \BE\Big[\int_{I_n^{(1)}} 1_{(s \leq R)}1_{(u_{m_2 - 1} \leq U_{s-} < u_{m_2})} 1_{(Z_{s-} = j)}\mu_{jk}(s, U_{s-}) \dd s \Big],
\end{equation*}
which we recognise as  $\mu_n^{jk}$. As for the second term, for large enough $n$, the integrand (which is positive) is bounded by $M 1_{(s \leq R)}1_{(u_{m_2 - 1} \leq U_{s-} < u_{m_2})} 1_{(Z_{s-} = j)}$ for some constant $M$ due to the continuity assumption on $\mu_{jk}$ and the fact that $\Delta_n^{(1)}$ and $\Delta_n^{(2)}$ both shrink to zero. Hence
\begin{align*}
    \BE[(2)] &\leq M^2 \BE\Big[\int_{I_n^{(1)}}\int_{I_n^{(1)}} 1_{(s_1, s_2 \leq R)}1_{(u_{m_2 - 1} \leq U(s_1-), U(s_2-) < u_{m_2})} 1_{(Z_{s_1-} = j)}1_{(Z_{s_2-} = j)}\dd s_1 \dd s_2 \Big] \\
    &= M^2 \int_{I_n^{(1)}}\int_{I_n^{(1)}} p_j^\texttt{c}(s_1, s_2, u_{m_2}) - p_j^\texttt{c}(s_1, s_2, u_{m_2 - 1}) \dd s_1 \dd s_2
\end{align*}
where, with slight abuse of notation, we have defined
\begin{equation*}
    p_j^\texttt{c}(s_1, s_2, v) = \BP(U(s_1), U(s_2) \leq v, s_1, s_2 \leq R, Z_{s_1-} = j, Z_{s_2-} = j).
\end{equation*}
Doing a second order Taylor expansion in the point $(t, t, u)$, we get
\begin{equation*}
    p_j^\texttt{c}(s_1, s_2, u_{m_2}) - p_j^\texttt{c}(s_1, s_2, u_{m_2 - 1}) = \Delta_n^{(2)} \partial_2p_j^\texttt{c}(t, u) + O\Big(\big(\Delta_n^{(1)}\big)^2 + \big(\Delta_n^{(2)}\big)^2\Big),
\end{equation*}
letting us conclude that $\BE[(2)] = O\Big(\big(\Delta_n^{(1)}\big)^2 \Delta_n^{(2)}\Big)$. As for (3), we bound the second integral in a very crude fashion by $M \Delta_n^{(1)}$ (for $n$ large enough). Then
\begin{equation*}
    \BE[(3)] \leq 2M\Delta_n^{(1)}\mu_n^{jk} = O\Big(\big(\Delta_n^{(1)}\big)^2 \Delta_n^{(2)}\Big)
\end{equation*}
which is of the same order as (2). The claim for $\operatorname{Var}(X_{1, n})$ now follows by applying 1. 
\item Finally, 
\begin{align*}
    \BE[Y_{1, n}^2] &= \int_{I_n^{(1)}}\int_{I_n^{(1)}} \BP(s_1, s_2 \leq R, u_{m_2 - 1} \leq U(s_1-) , U(s_2-) < u_{m_2}, Z_{s_1-} = j, Z_{s_2-} = j) \dd s_1 \dd s_2 \\
    &= \int_{I_n^{(1)}}\int_{I_n^{(1)}} p_j^\texttt{c}(s_1, s_2, u_{m_2}) - p_j^\texttt{c}(s_1, s_2, u_{m_2 - 1})\dd s_1 \dd s_2
\end{align*}
with $p_j^\texttt{c}(s_1, s_2, v)$ defined in the proof of 3. Now apply a second order Taylor expansion around $v = u$ to get
\begin{equation*}
    p_j^\texttt{c}(s_1, s_2, u_{m_2}) - p_j^\texttt{c}(s_1, s_2, u_{m_2 - 1}) = \partial_2p_j^\texttt{c}(s_1, s_2, u)\Delta_n^{(2)} + O\Big(\big(\Delta_n^{(2)}\big)^2\Big),
\end{equation*}
yielding
\begin{equation*}
    \BE[Y_{1, n}^2] = \Delta_n^{(2)}\int_{I_n^{(1)}}\int_{I_n^{(1)}} \partial_2p_j^\texttt{c}(s_1, s_2, u)\dd s_1 \dd s_2 + O(\big(\Delta_n^{(1)}\big)^2 \big(\Delta_n^{(2)}\big)^2).
\end{equation*}
Now apply a first order Taylor expansion around $(t, t)$ to get
\begin{equation*}
    \BE[Y_{1, n}^2] = \big(\Delta_n^{(1)}\big)^2 \Delta_n^{(2)} \partial_2p_j^\texttt{c}(t, u) + O\Big(\big(\Delta_n^{(1)}\big)^2 \big(\Delta_n^{(2)}\big)^2\Big) + O\Big(\big(\Delta_n^{(1)}\big)^3 \Delta_n^{(2)}\Big).
\end{equation*}
Since $\eta_n^j = O\Big(\big(\Delta_n^{(1)}\big)^2 \big(\Delta_n^{(2)}\big)^2\Big)$, this asymptotic result also applies to the variance, and the proof is complete. 
\end{enumerate}
\end{proof}

\begin{proof}[Proof of Lemma \ref{lem:LyapunovSemiMarkov}]
Start by decomposing
\begin{align*}
    &X_{i, n} - \mu_n^{jk} = \int_{I_n^{(1)}} 1_{(s \leq R^i)}1_{(u_{m_2 - 1} \leq U_{s-}^i < u_{m_2})} \dd M_{jk}^i(s) \\
    &+ \int_{I_n^{(1)}}\Big( 1_{(s \leq R^i)}1_{(u_{m_2 - 1} \leq U_{s-}^i < u_{m_2})} 1_{(Z_{s-}^i = j)}\mu_{jk}(s, U_{s-}^i) - \int_{I_n^{(2)}} \mu_{jk}(s, v) \partial_2p_j^\texttt{c}(s, v) \dd v\Big) \dd s \\
    &=: (I) + (II).
\end{align*}
We apply a Minkowski inequality in order to bound $\BE[(X_{i, n} - \mu_n^{jk})^4]$. As for the first term, note that $s \mapsto 1_{(s \leq R^i)}1_{(u_{m_2 - 1} \leq U_{s-}^i < u_{m_2})}$ is predictable. Hence by Lemma \ref{lem:generalfourthmoment}, 
\begin{equation*}
    \BE[(I)^4] = \int_{I_n^{(1)}} \BE\Big[(6(I)^2 + (I) + 1)1_{(s \leq R^i)}1_{(u_{m_2 - 1} \leq U_{s-}^i < u_{m_2})} 1_{(Z^i_{s-} = j)}\mu_{jk}(s, U_{s-}^i)\Big]\dd s.
\end{equation*}
Apply a first order Taylor expansion to conclude that $\BE[(I)^4] = O\big(\Delta_n^{(1)}\Delta_n^{(2)}\big)$. As for (II), denote
\begin{equation*}
    A_n := \int_{I_n^{(1)}} 1_{(s \leq R^i)}1_{(u_{m_2 - 1} \leq U_{s-}^i < u_{m_2})}1_{(Z_{s-}^i = j)}\mu_{jk}(s, U_{s-}^i)\dd s
\end{equation*}
so that $(II) = A_n - \BE[A_n]$. Since $A_n \geq 0$, we have
\begin{align*}
    \BE[(II)^4] &= \BE[A_n^4] - 4\BE[A_n^4]\BE[A_n] + 6\BE[A_n^2]\BE[A_n]^2 - 3\BE[A_n]^4 \\
    &\leq \BE[A_n^4] + 6\BE[A_n^2] \BE[A_n]^2.
\end{align*}
Using standard Taylor arguments, it is not difficult to see that $\BE[A_n^4] = O\Big(\big(\Delta_n^{(1)} \big)^4 \Delta_n^{(2)}\Big)$ and $\BE[A_n^2] = O\Big(\big(\Delta_n^{(1)} \big)^2 \Delta_n^{(2)}\Big)$. Since $\BE[A_n] = O(\Delta_n^{(1)} \Delta_n^{(2)})$, we get $\BE[(II)^4] = O\Big(\big(\Delta_n^{(1)}\big)^4 \Delta_n^{(2)} \Big)$. All in all, from Minkowski's inequality,
\begin{align*}
    \BE[(X_{i, n} - \mu_n^{jk})^4]^{1/4} &= O\Big(\big(\Delta_n^{(1)} \big)^{1/4}\big(\Delta_n^{(2)} \big)^{1/4} \Big) + O\Big(\Delta_n^{(1)} \big(\Delta_n^{(2)}\big)^{1/4} \Big) \\
    &= O\Big(\big(\Delta_n^{(1)} \big)^{1/4}\big(\Delta_n^{(2)} \big)^{1/4} \Big),
\end{align*}
so
\begin{equation*}
    \BE\left[\left(\frac{X_{i, n} - \mu_n^{jk}}{\sqrt{\operatorname{Var}(\sum_{i = 1} X_{i, n})}} \right)^4 \right] = O\left(\frac{\Delta_n^{(1)}\Delta_n^{(2)}}{n^2 \big(\Delta_n^{(1)}\big)^2\big(\Delta_n^{(2)}\big)^2} \right) = O\left(\frac{1}{n} \cdot \frac{1}{n\Delta_n^{(1)}\Delta_n^{(2)}} \right),
\end{equation*}
completing the proof.
\end{proof}

\begin{proof}[Proof of Corollary \ref{cor:asymptoticIndependenceSemiMarkov}]
Analogous to the Markov case, it suffices to show weak convergence of
\begin{equation*}
    \lambda_1 \frac{O_{jk}(m_1(t), m_2(u)) - n\mu_n^{jk}(t, u)}{\sqrt{n\Delta_n^{(1)}\Delta_n^{(2)}} \partial_2 p_j^\texttt{c}(t, u)} + \lambda_2 \frac{O_{jk}(m_1(s), m_2(v)) - n\mu_n^{jk}(s, v)}{\sqrt{n\Delta_n^{(1)}\Delta_n^{(2)}} \partial_2 p_j^\texttt{c}(s, v)}
\end{equation*}
where $\lambda_1, \lambda_2 \in \BR$ are arbitrary. We can write this as $\sum_{i = 1}^n W_{ni}(t, u, s, v)$ with
\begin{equation*}
    W_{ni}(t, u, s, v) := \lambda_1 \frac{X_{i, n}(t, u) - \mu_n^{jk}(t, u)}{\sqrt{n\Delta_n^{(1)}\Delta_n^{(2)}} \partial_2 p_j^\texttt{c}(t, u)} + \lambda_2 \frac{X_{i, n}(s, v) - \mu_n^{jk}(s, v)}{\sqrt{n\Delta_n^{(1)}\Delta_n^{(2)}} \partial_2 p_j^\texttt{c}(s, v)}.
\end{equation*}
From Lemma \ref{lem:meanvariancesemiMarkov} and the same arguments as in the Markov case, it follows that
\begin{equation*}
    \sum_{i = 1}^n \BE[W_{ni}^4(t, u, s, v)] = O\left(\frac{1}{n\Delta_n^{(1)}\Delta_n^{(2)}}\right)
\end{equation*}
so that the Lyapunov condition of order 4 holds. Again by Lemma \ref{lem:meanvariancesemiMarkov},
\begin{align*}
    \sum_{i = 1}^n \Var(W_{ni}(t, u, s, v)) &= \lambda_1^2 \frac{\mu_{jk}(t, u)}{\partial_2 p_j^\texttt{c}(t, u)} + \lambda_2^2 \frac{\mu_{jk}(s, v)}{\partial_2 p_j^\texttt{c}(s, v)} + O(\Delta_n^{(1)}) + O(\Delta_n^{(2)}) \\
    &+ \frac{2\lambda_1 \lambda_2}{\Delta_n^{(1)}\Delta_n^{(2)}p_j^\texttt{c}(t, u)\partial_2 p_j^\texttt{c}(s, v)} \Cov(X_{i, n}(t, u), X_{i, n}(s, v)).
\end{align*}
It remains to compute the covariance. Just like for Markov,
\begin{align*}
    X_{i, n}(t, u) &= \int_{I_n^{(1)}(t)} 1_{(w \leq R^i)} 1_{(u_{m_2(u) - 1} \leq U_{w-}^i < u_{m_2(u)})} \dd M_{jk}^i(w) \\
    &+ \int_{I_n^{(1)}(t)} 1_{(w \leq R^i)} 1_{(u_{m_2(u) - 1} \leq U_{w-}^i < u_{m_2(u)})} 1_{(Z_{w-}^i = j)}\mu_{jk}(w, U_{w-}^i)\dd w.
\end{align*}
The product $X_{i, n}(t, u)X_{i, n}(s, v)$ again yields four terms, the last three being of order $O\big(\big(\Delta_n^{(1)}\big)^2 \Delta_n^{(2)}\big)$. For large enough $n$, the intervals $I_n^{(1)}(t)$ and $I_n^{(1)}(s)$ are disjoint, and the same conditioning argument with the $\mathcal{G}$-filtration and the martingale property yields that the first term has mean zero. We conclude $\Cov(X_{i, n}(t, u), X_{i, n}(s, v)) = O\big(\big(\Delta_n^{(1)} \big)^2 \Delta_n^{(2)} \big)$ so that
\begin{equation*}
    \sum_{i = 1}^n \Var(W_{ni}(t, u, s, v)) = \lambda_1^2 \frac{\mu_{jk}(t, u)}{\partial_2 p_j^\texttt{c}(t, u)} + \lambda_2^2 \frac{\mu_{jk}(s, v)}{\partial_2 p_j^\texttt{c}(s, v)} + O(\Delta_n^{(1)}) + O(\Delta_n^{(2)}),
\end{equation*}
and the proof is complete by the Cramér--Wold device as before. 
\end{proof}

\section{Background on counting processes}\label{sec:background}

This appendix is for useful background results on counting processes that will be used in the proofs of our results. Let $N = (N(t))_{t \geq 0}$ denote a counting process with predictable compensator $\Lambda$. We say that $N$ has a (predictable) intensity process $\lambda$ if we can write
\begin{equation*}
    \Lambda(t) = \int_0^t \lambda(s)\dd s, \quad t \geq 0.
\end{equation*}
In the following, we let $M(t) = N(t) - \Lambda(t)$ denote the corresponding counting process martingale. If $M$ is square-integrable and $N$ has intensity process $\lambda$, it follows from (2.4.3) in \cite{ABGK} that
\begin{equation}\label{eq:quadraticVariation}
	M(t)^2 - \int_0^t \lambda(s)\mathrm{d}s
\end{equation}
is a mean-zero martingale. Using the exact same argument, one can verify the same result for the shifted process $\tilde{M}(t) = M(t) - M(s)$ where $s \geq 0$ is a fixed time. We state this as a lemma.

\begin{lem}\label{lem:martingaleshiftedsecondmoment}
If $N$ has intensity $\lambda$ and $s \geq 0$ is fixed, the process $\tilde{M}(t) = M(t) - M(s)$ for $t \geq s$ is a mean-zero martingale and
\begin{equation*}
	\BE[(M(t) - M(s))^2] = \BE\Big[\int_s^t \lambda(u)\mathrm{d}u\Big].
\end{equation*}
\end{lem}

The lemma can also be proved by straightforward calculations using (\ref{eq:quadraticVariation}). The following result is useful for verifying certain Lyapunov conditions.

\begin{lem}\label{lem:generalfourthmoment}
Let $N$ be a counting process with $\BE[N(t)^4] < \infty$ and intensity $\lambda$, and let $M$ denote the corresponding martingale. If $H$ is a predictable process, then it holds for the process
\begin{equation*}
    X(t) = \int_0^t H(s)\mathrm{d}M(s)
\end{equation*}
that
\begin{equation*}
    \BE[X(t)^4] = \int_0^t \BE[(6X(s)^2 H(s)^2 + 4X(s)H(s)^3 + H(s)^4)\lambda(s)]\dd s
\end{equation*}
and
\begin{equation*}
    \BE[(X(t) - X(s))^4] = \int_s^t \BE[(6(X(u) - X(s))^2 H(u)^2 + 4(X(u) - X(s))H(u)^3 + H(u)^4)\lambda(u)]\dd u.
\end{equation*}
\end{lem}
\begin{proof}
Applying the Itô formula (see e.g. Theorem 14.2.3 in \cite{CohenElliott}), we get for any $n \in \BN$ that
\begin{equation*}
    \dd(X(t)^n) = -n\lambda(t)H(t)X(t)^{n - 1}\dd t + (X(t-) + \Delta X(t))^n - X(t-)^n.
\end{equation*}
Since $\Delta X(t) = H(t)\Delta M(t) = H(t)\Delta N(t)$, we have for $n = 4$ that
\begin{align*}
    (X(t-) + \Delta X(t))^4 - X(t-)^4 &= 4X(t-)^3 H(t)\Delta N(t) + 6X(t-)^2 H(t)^2 \Delta N(t) \\
    &+ 4X(t-) H(t)^3 \Delta N(t) + H(t)^4 \Delta N(t).
\end{align*}
Thus,
\begin{align*}
    \dd(X(t)^4) &= -4\lambda(t)H(t)X(t)^3 \dd t \\
    &+ (4X(t-)^3 H(t) + 6X(t-)^2 H(t)^2 + 4X(t-) H(t)^3 + H(t)^4)\dd N(t).
\end{align*}
Adding and subtracting the compensator from $\dd N$, we get
\begin{align*}
    X(t)^4 &= \int_0^t (6X(s)^2 H(s)^2 + 4X(s)H(s)^3 + H(s)^4)\lambda(s) \dd s \\
    &+ \int_0^t (4X(s-)^3 H(s) + 6X(s-)^2 H(s)^2 + 4X(s-)H(s)^3 + H(s)^4)\dd M(s).
\end{align*}
The second term is a martingale, and the first assertion follows immediately. As for the second claim, note that if $s \geq 0$ is fixed, it holds for the process $\tilde{X}(t) = X(t) - X(s)$ that $\dd \tilde{X}(t) = \dd X(t)$ and $\Delta \tilde{X}(t) = \Delta X(t)$. Hence we can copy the proof of the $s = 0$ case.
\end{proof}

\begin{cor}\label{cor:shiftedfourthmoment}
For a counting process $N$ with $\BE[N(t)^4] < \infty$ and intensity $\lambda$, it holds for $t \geq s \geq 0$ and $\tilde{M}(t) = M(t) - M(s)$ that
\begin{align*}
    \BE[(M(t) - M(s))^4] = \int_s^t \BE[(6\tilde{M}(u-)^2 + 4\tilde{M}(u-) + 1)\lambda(u)]\dd u.
\end{align*}
\end{cor}

\end{document}